\documentstyle[twoside,amsmath, amssymb, mathrsfs, amsfonts, eucal, epsfig, 11pt]{article}

 \newtheorem{theorem}{Theorem}[section]
 \newtheorem{lemma}[theorem]{Lemma}
 \newtheorem{proposition}[theorem]{Proposition}
 \newtheorem{corollary}[theorem]{Corollary}
 \newtheorem{definition}[theorem]{Definition}
 \newenvironment{proof}{\begin{trivlist} \item[]{\em Proof.}}{\end{trivlist}}
 \newcount\refno
\def\CC{{\mathbb C}}
\def\DD{{\mathbb D}}

\def\BB{{\mathbb B}}
 
 \def\RR{{\mathbb R}}
 \def\NN{{\mathbb N}}
 \def\SB{{\mathbb S}}
 
 \def\ZZ{{\mathbb Z}}

 \def\SF{{\mathscr F}}

 \title{\bf Local boundary behaviour and the area integral of generalized harmonic functions associated with root systems
 \thanks{{Supported by the National Natural
 Science Foundation of China (No. 12071295).}
 \newline
 \indent \,\,$^\dag$Corresponding author.
 \newline
 \indent\,\, E-mail: jiujiaxi78@163.com (J.-X. Jiu); lizk@shnu.edu.cn (Zh.-K. Li).}}

\author{{Jiaxi Jiu and Zhongkai Li$^{\dag}$}\\
{\small Department of Mathematics, Shanghai Normal University, Shanghai 200234, China}
}

\date{\it In memory of Professor Leetsch Charles Hsu (Lizhi Xu)}

\begin{document}
 \maketitle \setcounter{page}{1} \pagestyle{myheadings}
 \markboth{Jiu and Li}{Generalized harmonic functions associated with root systems}

 \begin{abstract}
 \noindent
The rational Dunkl operators are commuting differential-reflection operators on the Euclidean space $\RR^d$ associated with a root system. The aim of the paper is to study local boundary behaviour of generalized harmonic functions associated with the Dunkl operators. We introduce a Lusin-type area integral operator $S$ by means of Dunkl's generalized translation and the Dunkl operators. The main results are on characterizations of local existence of non-tangential boundary limits of a generalized harmonic function $u$ in the upper half-space $\RR^{d+1}_+$ associated with the Dunkl operators, and for a subset $E$ of $\RR^d$ invariant under the reflection group generated by the root system, the equivalence of the following three assertions are proved: (i) $u$ has a finite non-tangential limit at $(x,0)$ for a.e. $x\in E$; (ii) $u$ is non-tangentially bounded for a.e. $x\in E$; (iii) $(Su)(x)$ is finite for a.e. $x\in E$.

 \vskip .2in
 \noindent
 {\bf 2020 MS Classification:} 	42B99(Primary), 31B25, 51F15(Secondary) 
 \end{abstract}

\setcounter{page}{1}

\section{Introduction and main results}

{\bf 1.1.}\quad A classical result of Fatou states that a bounded harmonic function $u$ in the unit disc $\DD$ has a non-tangential limit at almost every point $e^{i\theta}$. A local version of Fatou's theorem proved by Privalov \cite{Pr1} is that if $u(z)$ is harmonic in $\DD$, and at each point $e^{i\theta}$ of a measurable subset $E$ of positive measure of the boundary $\partial\DD$, there is a cone contained in $\DD$ with vertex $e^{i\theta}$ in which $u(z)$ is bounded, then $u$ has a non-tangential limit at almost every $e^{i\theta}\in E$. Marcinkiewicz and Zygmund \cite{MZ1}, and Spencer \cite{Sp1} obtained a completely different criterion on existence of non-tangential limit at almost every point $e^{i\theta}$ of a measurable set $E\subset\partial\DD$ which is characterized by finiteness of Lusin's area integral $\int\left|\nabla u(x,y)\right|^2 dxdy$ for almost every $e^{i\theta}\in E$, where the integral is taken over a cone contained in $\DD$ with vertex $e^{i\theta}$. One of the basic tools in these studies is the conformal mapping, which introduces technical difficulties in extending them to more variables and other settings. The breakthrough to this obstacle was made by Calder${\rm \acute{o}}$n \cite{Cal1,Cal2}, who generalized Privalov's theorem and Marcinkiewicz and Zygmund's theorem to Euclidean half-spaces of several variables by real method. Stein in \cite{St1} extended Spencer's theorem to several variables. (For a thoughtful treatment on these results, see \cite{St2}.)
Since then, criteria on local existence of non-tangential boundary limits of harmonic functions in many different contexts, in terms of non-tangential boundedness or one-side non-tangential boundedness or finiteness of non-tangential area integral have been intensively studied (see, among others, \cite{AP1}-\cite{Br2}, \cite{Car1}, \cite{Gu1}, \cite{HW1}-\cite{JK1}, \cite{Ko1}-\cite{Kr2}, \cite{Le1}, \cite{Mo1}-\cite{Mo3}, \cite{Pe1}-\cite{Pi1}, \cite{Pu1}, \cite{Wi1}).
Various aspects associated with the Fatou theory in one and several complex variables are contained in the most recent paper \cite{DbK}.

The rational Dunkl operators are differential-reflection operators on the Euclidean space $\RR^d$ associated with a root system. In this paper we study local boundary behaviour of generalized harmonic functions in the upper half-space $\RR^{d+1}_+=\RR^d\times(0,\infty)$ associated with the Dunkl operators. As usual we denote by $e_1,e_2,\dots,e_d$ the standard basis vectors in $\RR^d$.
\vskip .2in

{\bf 1.2.}\quad Let $R$ be a root system in $\RR^d$ normalized so that $\langle\alpha,\alpha\rangle=2$ for $\alpha\in R$ and with $R_+$ a fixed positive subsystem, and $G$ the finite reflection group generated by the reflections $\sigma_\alpha$ ($\alpha\in R$), where $\sigma_{\alpha}(x)=x-\langle\alpha,x\rangle\alpha$ for $x\in\RR^d$. For a multiplicity function $\kappa$ defined on $R$ (invariant under $G$), the Dunkl operators $D_j$ ($1\le j\le d)$ are defined by (cf. \cite{Du2})
\begin{eqnarray}\label{Dunkl-operator-1}
D_j=\partial_j+\sum_{\alpha\in R_+}\kappa(\alpha)\alpha_j\frac{1-\sigma_{\alpha}}{\langle\alpha,x\rangle},
 \end{eqnarray}
where $\sigma_{\alpha}f=f\circ\sigma_{\alpha}$, $\alpha_j$ is the $j$th coordinate of $\alpha$, and the associated measure is given by
\begin{eqnarray*}
d\omega_{\kappa}=W_{\kappa}(x)dx,\qquad{\rm with}\quad  W_{\kappa}(x)=\prod_{\alpha\in R_+}|\langle\alpha,x\rangle|^{2\kappa(\alpha)}.
\end{eqnarray*}
The weight $W_{\kappa}$ is a positive homogeneous function of degree $2|\kappa|$, where $|\kappa|=\sum_{\alpha\in R_+}k(\alpha)$.
It is proved in \cite{Du2} that the Dunkl operators commute with each other.

A $C^2$ function $u$ defined in a $G$-invariant domain $\Omega$ of $\RR^{d+1}$ is said to be $\kappa$-harmonic if $\Delta_{\kappa}u=0$, where $\Delta_{\kappa}$ is the $\kappa$-Laplacian defined by
\begin{align*}
\Delta_{\kappa}=\partial_{y}^2+\sum_{j=1}^{d}D_{j}^{2}.
\end{align*}
From \cite{Du1} the $\kappa$-Laplacian $\Delta_{\kappa}$ can be explicitly expressed by
\begin{align}\label{Laplace-2-1}
\Delta_{\kappa}u=\Delta u+2\sum_{\alpha\in R_+}\kappa(\alpha)\delta_\alpha u,
\end{align}
where
\begin{align}\label{Laplace-2-2}
\delta_\alpha u=\frac{\langle\nabla^{(x)} u,\alpha\rangle}{\langle\alpha,x\rangle}-\frac{u-\sigma_{\alpha}u}
{\langle\alpha,x\rangle^2},
\end{align}
$\Delta=\partial_{y}^2+\sum_{j=1}^{d}\partial_{j}^{2}$ is the usual Laplacian on $\RR^{d+1}$, $\nabla^{(x)}$ the usual gradient operator on $\RR^{d}$, and $(\sigma u)(x,y)=u(\sigma(x),y)$ for $\sigma\in G$ and $(x,y)\in\RR^{d+1}$.
By \cite[Proposition 1.1]{Du1}, if $u$ is $\kappa$-harmonic, so is $\sigma u$ for each $\sigma\in G$.

Several examples are given as follows, see \cite{Ro5}.

{\bf Examples.}\quad  {\bf 1.} The reflection group generated by the root system $R=\{\pm\sqrt{2}e_j:\,1\le j\le d\}$ is denoted by $Z_2^d$. The associated measure is $d\omega_{\lambda}=\prod_{j=1}^d|x_j|^{2\lambda_j}dx$ with multiplicity parameters $\lambda=(\lambda_1,\dots,\lambda_d)$, and the Dunkl operators are given by $D_j=\partial_j+\frac{\lambda_j}{x_j}(1-\sigma_j)$, where $\sigma_j(x)=(x_1,\dots, -x_j,\dots,x_d)$. In this special case, we shall write $\Delta_{\lambda}$ instead of $\Delta_{\kappa}$, and say ``$\lambda$-harmonic" instead of ``$\kappa$-harmonic".

{\bf 2.} The type $A_{d-1}$ case. A root system of $S_d$, the symmetric group in $d$ elements, is given by $R=\{\pm(e_i-e_j):\,1\le i<j\le d\}$. Since $S_d$ has only one orbit in $R$, a multiplicity function is constant. The Dunkl operators for a given $\kappa\in\CC$ have the form
\begin{eqnarray*}
D_j=\partial_j+\kappa\sum_{i\neq j}\frac{1-\sigma_{ij}}{x_j-x_i},
 \end{eqnarray*}
where $\sigma_{ij}(\dots,x_i,\dots,x_j,\dots)=(\dots,x_j,\dots,x_i,\dots)$.

{\bf 3.} The type $B_{d}$ case. The root system corresponding to the semidirect product $G=S_d\ltimes Z_2^d$ is given by $R=\{\pm\sqrt{2}e_j:\,1\le j\le d;\,\,\pm(e_i\pm e_j):\,1\le i<j\le d\}$. A multiplicity function has the form $\kappa=(\kappa_0,\kappa_1)$, with only two independent values, as there are two conjugacy classes of reflections. The associated Dunkl operators are given by
\begin{eqnarray*}
D_j=\partial_j+\frac{\kappa_0}{x_j}(1-\sigma_j)+\kappa_1\sum_{i\neq j}\left[\frac{1-\sigma_{ij}}{x_j-x_i}+\frac{1-\tau_{ij}}{x_j+x_i}\right],
 \end{eqnarray*}
where $\tau_{ij}=\sigma_{ij}\sigma_i\sigma_j$.

During the past three decades, the Dunkl operators have gained considerable interest in various fields of
mathematics and in physical applications; they are, in the case of the symmetric group $S_d$, naturally connected
with the Calogero-Sutherland models of quantum many body systems (see \cite{LV}, for example).

In what follows, we assume that the involved multiplicity functions $\kappa$ are real and nonnegative. As usual, for a subset $E$ of $\RR^d$, define $\sigma E=\{\sigma(x):\,x\in E\}$ for $\sigma\in G$, and $E$ is called to be $G$-invariant if $\sigma E=E$ for all $\sigma\in G$; for a function $f$ defined on a $G$-invariant subset $E$ of $\RR^d$, set $\sigma f=f\circ\sigma$ for $\sigma\in G$, and $f$ is called to be $G$-invariant on $E$ if $\sigma f=f$ for all $\sigma\in G$.


\vskip .2in

{\bf 1.3. Main results.}\quad We are concerned with characterizations on existence of non-tangential boundary values, in a local sense, of $\kappa$-harmonic functions in the upper half-space $\RR^{d+1}_+$.

Our first result is to extend those of Privalov \cite{Pr1} and Calder${\rm \acute{o}}$n \cite{Cal1} to the $\kappa$-harmonic functions, which characterizes the local existence of non-tangential boundary values by non-tangential boundedness. As usual, we denote by $\Gamma_a(x)$ the positive cone of aperture $a>0$ with vertex $(x,0)\in \partial\RR^{d+1}_+=\RR^d$, and $\Gamma^h_a(x)$ the truncated cone with height $h>0$, that is,
$$
\Gamma^h_a(x)=\left\{(t,y):\quad |t-x|<ay \quad\hbox{for}\,\,t\in\RR^d,\,\,0<y<h\right\}.
$$

\begin{definition}
For a given function $u$ defined in $\RR^{d+1}_+$ and for $x\in\RR^d$,

{\rm (i)} that $u$ has a non-tangential limit at $(x,0)$ means that for every $a>0$, $\lim u(t,y)$ exists as $(t,y)\in\Gamma_a(x)$ approaching to $(x,0)$; and

{\rm (ii)} that $u$ is said to be non-tangentially bounded at $(x,0)$ if $u(t,y)$ is bounded in $\Gamma^h_a(x)$ for some $a,h>0$.
\end{definition}

\begin{theorem}\label{thm-1}
Suppose that $E$ is a $G$-invariant measurable subset of $\partial\RR^{d+1}_+=\RR^d$ and $u$ is $\kappa$-harmonic in $\RR^{d+1}_+$. Then the following assertions are equivalent:

{\rm (i)}\, $u$ has a finite non-tangential limit at $(x,0)$ for almost every $x\in E$;

{\rm (ii)} $u$ is non-tangentially bounded at $(x,0)$ for almost every $x\in E$.
\end{theorem}

Our next attempt is to generalize the results of Marcinkiewicz and Zygmund \cite{MZ1} and Calder${\rm \acute{o}}$n \cite{Cal2} to the $\kappa$-harmonic functions. For a $C^2$ function $u$ in $\RR^{d+1}_+$, we introduce, at first formally, a Lusin-type area integral $Su=S_{a,h}u$ for some $a,h>0$ in the Dunkl setting, that is,
\begin{align}\label{area-integral-1-1}
\left(S_{a,h}u\right)(x)=\left(\iint_{\Gamma_{a}^{h}(0)}
\left[\tau_{x}(\Delta_{\kappa}u^{2})\right](t,y)\,y^{1-d-2|\kappa|}\,d\omega_{\kappa}(t)dy\right)^{1/2},
\end{align}
where $\tau_x$ is the associated (generalized) translation in the Dunkl setting (see Section 2), acting on the first argument. We note that $S_{a,h}u$ is first defined in Liao \cite{Liao1} for $G=Z_2$ with $d=1$, where it is used to characterize the Hardy space $H^1_{\lambda}(\RR)$ associated to the rank-one Dunkl operator.

\begin{theorem}\label{thm-2}
Suppose that $E$ is a $G$-invariant measurable subset of $\partial\RR^{d+1}_+=\RR^d$ and $u$ is $\kappa$-harmonic in $\RR^{d+1}_+$. If $u$ is non-tangentially bounded at $(x,0)$ for every $x\in E$, then for almost every $x\in E$, the area integral $(S_{a,h}u)(x)$ is finite for some $a,h>0$.
\end{theorem}

Finally we consider the  converse of Theorem \ref{thm-2}, which is a generalization of the results of Spencer \cite{Sp1} and Stein \cite{St1}. The answer is positive when $u$ is $G$-invariant (i.e. $u(\sigma(x),y)=u(x,y)$ for all $\sigma\in G$).

\begin{theorem}\label{thm-3}
Suppose that $E$ is a $G$-invariant measurable subset of $\partial\RR^{d+1}_+=\RR^d$ and $u$ is $\kappa$-harmonic in $\RR^{d+1}_+$. If $u$ is $G$-invariant and for every $x\in E$,
its area integral $(S_{a,h}u)(x)$ is finite for some $a,h>0$, then $u$ is non-tangentially bounded at $(x,0)$ for almost every $x\in E$.
\end{theorem}

From Theorems \ref{thm-1} to \ref{thm-3}, we have

\begin{corollary}\label{cor-3}
Suppose that $E$ is a $G$-invariant measurable subset of $\partial\RR^{d+1}_+=\RR^d$ and $u$ is $\kappa$-harmonic in $\RR^{d+1}_+$. If $u$ is $G$-invariant, then the following assertions are equivalent:

{\rm (i)}\,\, $u$ has a finite non-tangential limit at $(x,0)$ for almost every $x\in E$;

{\rm (ii)}\, $u$ is non-tangentially bounded at $(x,0)$ for almost every $x\in E$;

{\rm (iii)} for almost every $x\in E$, the area integral $(S_{a,h}u)(x)$ is finite for some $a,h>0$.
\end{corollary}

For the group $G=Z_{2}^{d}$, the $G$-invariance of $u$ can be eliminated. Indeed we have

\begin{theorem}\label{thm-4}
Suppose $G=Z_{2}^{d}$ with a given multiplicity parameters $\lambda=(\lambda_1,\dots,\lambda_d)$, $E$ is a $G$-invariant measurable subset of $\RR^d$, and $u$ is $\lambda$-harmonic in $\RR^{d+1}_+$.
If  for every $x\in E$, the area integral $(S_{a,h}u)(x)$ is finite for some $a,h>0$, then $u$ is non-tangentially bounded at $(x,0)$ for almost every $x\in E$.
\end{theorem}

Combining the conclusions in Theorems \ref{thm-1}, \ref{thm-2} and \ref{thm-4}, the following corollary is immediate.

\begin{corollary}\label{cor-5}
Suppose $G=Z_{2}^{d}$ with a given multiplicity parameters $\lambda=(\lambda_1,\dots,\lambda_d)$. If $E$ is a $G$-invariant measurable subset of $\RR^d$ and $u$ is $\lambda$-harmonic in $\RR^{d+1}_+$, then the following assertions are equivalent:

{\rm (i)}\,\, $u$ has a finite non-tangential limit at $(x,0)$ for almost every $x\in E$;

{\rm (ii)}\, $u$ is non-tangentially bounded at $(x,0)$ for almost every $x\in E$;

{\rm (iii)} for almost every $x\in E$, the area integral $(S_{a,h}u)(x)$ is finite for some $a,h>0$.
\end{corollary}
\vskip .2in

{\bf 1.4. Remarks.}\quad (a) We remark that the area integral $(S_{a,h}u)(x)$ defined in (\ref{area-integral-1-1}) is meaningful for $\kappa$-harmonic function $u$ in $\RR^{d+1}_+$ based upon the following three points: (i) $u\in C^{\infty}(\RR^{d+1}_+)$ by Proposition \ref{harmonic-2-b}; (ii) $\left(\Delta_{\kappa} u^2\right)(x,y)\ge0$ by Lemma \ref{Laplace-4-a}; (iii) the integral on the right hand side of (\ref{area-integral-1-1}) is nonnegative, since it can be written as
\begin{align*}
\lim_{\delta\rightarrow0+}\iint_{\substack{|t|\le ay, \\
y\in[\delta,h]}}
=\lim_{\delta\rightarrow0+}\int_{\delta}^h\int_0^{ay}M_{(\Delta_{\kappa}u^2)(\cdot,y)}(x,r)
\left(\frac{r}{y}\right)^{2|\kappa|+d-1}\,drdy,
\end{align*}
where $M_f(x,r)$ is the generalized spherical mean of $f$ (see (\ref{spherical-mean-2-1})) which is positivity-preserving.

(b) In the proofs of Theorems \ref{thm-2}, \ref{thm-3} and \ref{thm-4}, to use the generalized translation $\tau_x$ legitimately we shall work with a variant of the area integral $(S_{a,h}u)(x)$ involving a smooth and compactly supported function $\psi$, that is,
\begin{align*}
(S^{\psi}_{a,h}u)(x)=\left(\iint_{\RR^d\times(0,h]}\left[\tau_{x}(\Delta_{\kappa}u^{2})\right](t,y)\,
\psi\left(\frac{t}{ay}\right)y^{1-d-2|\kappa|}\,d\omega_{\kappa}(t)dy\right)^{1/2},
\end{align*}
where $\psi\in C^{\infty}(\RR^d)$ is radial and satisfies
$0\le\psi(x)\le1$,
\begin{align*}
\psi(x)=1\quad \hbox{for}\,\,|x|\le1/2,\quad\hbox{and}\quad \psi(x)=0\quad \hbox{for}\,\,|x|\ge1.
\end{align*}


(c) We note that the assumption of reflection-symmetry ($G$-invariance) on the given subset $E$ of $\RR^d$ in the above theorems and corollaries is reasonable, since the Dunkl operators involve values of functions at reflection-symmetric points. Certainly, when regarding local properties, it would be worthy of finding weaker requirements at reflection-symmetric points then that in the position under consideration. This is still open on the non-tangential behaviour of $\kappa$-harmonic functions.

(d) In comparison with the classical one, it seems that the more apt form of Lusin-type area integral in the Dunkl setting would be, for $a,h>0$,
\begin{align*}
\left(\iint_{\Gamma_{a}^{h}(0)}
\left[\tau_{x}\left(|\nabla u|^2\right)\right](t,y)\frac{d\omega_{\kappa}(t)dy}{y^{2|\kappa|+d-1}}\,\right)^{1/2}\,\,\hbox{or}\,\,
\left(\iint_{\Gamma_{a}^{h}(0)}
\left[\tau_{x}\left(|\nabla_{\kappa} u|^2\right)\right](t,y)\frac{d\omega_{\kappa}(t)dy}{y^{2|\kappa|+d-1}}\,\right)^{1/2},
\end{align*}
where $\nabla=\nabla^{(x,y)}$ is the $(d+1)$-dimensional gradient and $\nabla_{\kappa}=(D_1,\dots,D_d,\partial_y)$ the one associated to the Dunkl operators (the $\kappa$-gradient).
Lemma \ref{Laplace-4-a} implies that both the two forms above are dominated by $(S_{a,h}u)(x)$, and for a $G$-invariant $u$, all three of them are consistent (up to a constant). Therefore Theorem \ref{thm-2} is deeper than expected, and
Theorem \ref{thm-4} and Corollary \ref{cor-5} motivate that the adoption of $S_{a,h}u$ would be also necessary for a non-$G$-invariant $u$ associated to a general reflection group $G$.

(e) In \cite{HW1,HW2}, Hunt and Wheeden studied the local boundary behaviour of usual harmonic functions in a bounded Lipschitz domain, and in \cite{JK1}, Jerison and Kenig extended the classical results to a larger class of domains named non-tangentially accessible domains. In the Dunkl setting, further researches on boundary behaviour of $\kappa$-harmonic functions may be considered in a reflection-symmetric Lipschitz domain or non-tangentially accessible domain.


The paper is organized as follows. Section 2 contains some basic knowledge on the rational Dunkl theory, and in Section 3 we prove several propositions which is necessary in the sequel. The proofs of Theorems \ref{thm-1} and \ref{thm-2} are given in Sections 4 and 5 separately, and Theorems \ref{thm-3} and \ref{thm-4} are proved in Section 6.

For $1\le p<\infty$, the space $L_{\kappa}^p(\RR^d)$ consists of all functions $f$ on $\RR^d$ such
that $\|f\|_{L_{\kappa}^p(\RR^d)}:=\left(c_{\kappa}\int_{\RR^d}|f|^p\,d\omega_{\kappa}\right)^{1/p}<+\infty$, where $c_{\kappa}^{-1}=\int_{\RR^d}e^{-|x|^2/2}d\omega_{\kappa}(x)$, and $\|f\|_{L^{\infty}}$ is given in the usual
way.
For a measurable set $E\subset\RR^d$, set
$|E|_{\kappa}=\int_{E}d\omega_{\kappa}$.  Throughout the paper, $c$, $c'$, $c_1$, $c_2$ and so on, denote
constants which may be different in different occurrences.
$X\asymp Y$ means that $X\le c_1Y$ and $Y\le c_2X$ for some fixed constants $c_1,c_2>0$.

%

\section{Some facts in the Dunkl theory}

In this section we give an account on results from the Dunkl theory which will be relevant for the sequel. Concerning root systems
and reflection groups, see \cite{Hu}. We shall use the notation $\BB(x,r)$ (or $\BB((x,y),r)$) to denote the open ball in $\RR^d$ (or $\RR^{d+1}$) with radius $r$ centered at $x\in\RR^d$ (or $(x,y)\in\RR^{d+1}$).

\subsection{The Dunkl kernel and the Dunkl transform}

As in the last section, for a root system $R$ let $G$ be the associated reflection group and $\kappa$ a given nonnegative multiplicity function. As shown in \cite{Du3}, there exists a unique degree-of-homogeneity-preserving linear
isomorphism $V_{\kappa}$ on polynomials, which intertwines the associated commutative algebra of Dunkl operators and the algebra of usual partial differential operators, namely,
$D_jV_{\kappa}=V_{\kappa}\partial_j$ for $1\le i\le d$.  $V_{\kappa}$ commutes with the group action of $G$. (For a thorough analysis on $V_{\kappa}$ with general $\kappa$, see \cite{DJO}.)
It was proved in \cite{Ro3} that $V_{\kappa}$ is positive on the space of polynomials, and moreover, for each $x\in\RR^d$ there exists a unique (Borel) probability measure $\mu^{\kappa}_{x}$ of $\RR^d$ such that
$$
V_{\kappa}f(x)=\int_{\RR^d}f(\xi)\,d\mu^{\kappa}_{x}(\xi),
$$
where $\mu^{\kappa}_{x}$ is of compact support with
\begin{align}\label{intertwining-support-1}
\hbox{supp}\,\mu^{\kappa}_{x}\subseteq\hbox{co}\,\{\sigma(x):\,\,\sigma\in G\},
\end{align}
the convex hull of the orbit of $x$ under $G$.
The intertwining operator $V_{\kappa}$ has an extension to $C^{\infty}(\RR^d)$ and establishes a homeomorphism of this space (see \cite{Tr1,Tr2}).

Associated with $G$ and $\kappa$, the Dunkl kernel is defined by
\begin{eqnarray*}
E_{\kappa}(x,z)=V_{\kappa}\left(e^{\langle\cdot,z\rangle}\right)(x)
=\int_{\RR^d}e^{\langle\xi,z\rangle}\,d\mu^{\kappa}_{x}(\xi), \qquad x\in\RR^d,\,\,z\in\CC^d.
\end{eqnarray*}
According to \cite{Op}, for fixed $z\in\CC^d$, $x\mapsto E_{\kappa}(x,z)$ may
be characterized as the unique analytic solution of the system
\begin{align}\label{Dunkl-kernel-eigenfunction-1}
D^x_{j}E_{\kappa}(x,z)=z_jE_{\kappa}(x,z), \qquad x\in\RR^d,\,\,\,j=1,\dots,d,
\end{align}
with the initial value $E_{\kappa}(0,z)=1$; and $E_{\kappa}$ has a unique holomorphic extension to $\CC^d\times\CC^d$ and is symmetric in its arguments (see \cite{dJ,Op}). Furthermore (cf. \cite{dJ,Du3,Ro3}), for $z,w\in\CC^d$, $\lambda\in\CC$ and $\sigma\in G$, one has
\begin{align}\label{Dunkl-kernel-2-3}
E_{\kappa}(\lambda z,w)=E_{\kappa}(z,\lambda w),\qquad E_{\kappa}(\sigma(z),\sigma(w))=E_{\kappa}(z,w),
\end{align}
and for $x\in\RR^d$, $z\in\CC^d$, and all multi-indices $\nu=(\nu_1,\dots,\nu_d)$,
\begin{align}\label{Dunkl-kernel-2-4}
\left|\partial_z^{\nu}E_{\kappa}(x,z)\right|\le|x|^{|\nu|}\max_{\sigma\in G}e^{{\rm Re}\,\langle\sigma(x),z\rangle}.
\end{align}

For $f\in L_{\kappa}^1(\RR^d)$, its Dunkl transform is defined by
\begin{eqnarray*}
\left(\SF_{\kappa}f\right)(\xi)=c_{\kappa}\int_{\RR^d}f(x)E_{\kappa}(-i\xi,x)\,d\omega_{\kappa}(x),\qquad \xi\in\RR^d.
\end{eqnarray*}
The Dunkl transform $\SF_{\kappa}$ commutes with the $G$-action, and shares many of the important properties with the
usual Fourier transform (see \cite{Du4,dJ,Ro3,RV1}), part of which are listed as follows.

\begin{proposition} \label{transform-a} {\rm(i)} If $f\in L_{\kappa}^1(\RR^d)$, then ${\SF}_{\kappa}f\in C_0({\RR^d})$ and $\|{\SF}_{\kappa}f\|_{\infty}\leq\|f\|_{L_{\kappa}^1}$.

{\rm (ii)} {\rm (Inversion)} \ If $f\in L_{\kappa}^1(\RR^d)$
such that $\SF_{\kappa}f\in L_{\kappa}^1(\RR^d)$, then
$f(x)=[\SF_{\kappa}(\SF_{\kappa}f)](-x)$.

{\rm(iii)} For $f\in{\mathscr S}({\RR^d})$ (the Schwartz space), we have $[\SF_{\kappa}(D_jf)](\xi)=i\xi_j(\SF_{\kappa}f)(\xi)$, $[\SF_{\kappa}(x_jf)](\xi)=i[D_j(\SF_{\kappa}f)](\xi)$ for
$\xi\in\RR^d$ and $1\le j\le d$; and $\SF_{\kappa}$ is a homeomorphism of ${\mathscr S}(\RR^d)$.

{\rm(iv)} {\rm (product formula)} For $f_1,f_2\in
L_{\kappa}^1(\RR^d)$, we have $\int_{\RR^d}f_1\cdot\SF_{\kappa}f_2\,d\omega_{\kappa}=\int_{\RR^d}f_2\cdot\SF_{\kappa}f_1\,d\omega_{\kappa}$.

{\rm (v)} {\rm (Plancherel)} There exists a unique extension of
$\SF_{\kappa}$ to $L_{\kappa}^2(\RR^d)$ with
$\|\SF_{\kappa}f\|_{L_{\kappa}^2}=\|f\|_{L_{\kappa}^2}$.

{\rm (vi)} If $f\in L_{\kappa}^1(\RR^d)$ is radial, then its Dunkl transform ${\SF}_{\kappa}f$ is also radial.
\end{proposition}

\subsection{The generalized translation}

For $x\in\RR^d$, the generalized translation $\tau_x$ is defined in $L_{\kappa}^2(\RR^d)$ by (see \cite{TX})
\begin{align}\label{translation-2-0}
\left[\SF_{\kappa}(\tau_xf)\right](\xi)=E_{\kappa}(i\xi,x)(\SF_{\kappa}f)(\xi).
\end{align}
By Proposition \ref{transform-a}(v), $\tau_x$ is well defined and satisfies $\|\tau_xf\|_{L_{\kappa}^2}\le\|f\|_{L_{\kappa}^2}$ in view of (\ref{Dunkl-kernel-2-4}).

If $f$ is in an appropriate subclass, for example, $A_{\kappa}(\RR^d)=\left\{f\in L_{\kappa}^1(\RR^d):\,\SF_{\kappa}f\in L_{\kappa}^1(\RR^d)\right\}$ (cf. \cite{TX}) or ${\mathscr S}(\RR^d)$, $(\tau_xf)(t)$ may be expressed pointwise by (see \cite{Ro2,TX})
\begin{align}\label{translation-2-2}
(\tau_x f)(t)=c_{\kappa}\int_{\RR^d}E_{\kappa}(i\xi,x)E_k(i\xi,t)(\SF_{\kappa}f)(\xi)\,d\omega_{\kappa}(\xi).
 \end{align}

\begin{proposition}\label{translation-2-a} If $f,g\in L_{\kappa}^2(\RR^d)$, then for $x\in\RR^d$,
\begin{align}\label{translation-2-1}
\int_{\RR^d}(\tau_xf)(t)g(t)\,d\omega_{\kappa}(t)=\int_{\RR^d}f(t)(\tau_{-x}g)(t)\,d\omega_{\kappa}(t),
 \end{align}
and both sides are continuous in $x$.
\end{proposition}

The equality (\ref{translation-2-1}) was proved in \cite{TX} for $f\in A_{\kappa}(\RR^d)$ and $g\in L_{\kappa}^2(\RR^d)\cap L^{\infty}(\RR^d)$. The general case, for $f,g\in L_{\kappa}^2(\RR^d)$, follows from a density argument.
On the continuity of $F(x):=\int_{\RR^d}\tau_xf\cdot g\,d\omega_{\kappa}$, for $x,x^0\in\RR^r$ one has
\begin{align*}
\left|F(x)-F(x^0)\right|
\le\|\tau_xf-\tau_{x^0}f\|_{L_{\kappa}^2}\|g\|_{L_{\kappa}^2}
=\left\|\left(E_{\kappa}(ix,\cdot)-E_{\kappa}(ix^0,\cdot)\right)\SF_{\kappa}f\right\|_{L_{\kappa}^2}\|g\|_{L_{\kappa}^2}
 \end{align*}
by Proposition \ref{transform-a}(v) and (\ref{translation-2-0}); and then, in view of (\ref{Dunkl-kernel-2-4}), Lebesgue's dominated convergence theorem yields the desired continuity.

In \cite{Tr2} an abstract form of $\tau_x$ is given by $(\tau_xf)(t)=V_{\kappa}^tV_{\kappa}^x\left[(V_{\kappa}^{-1}f)(x+t)\right]$ for $f\in C^{\infty}(\RR^d)$ and $x,t\in\RR^d$.

\begin{proposition}\label{translation-2-b} {\rm(\cite{Tr2})}
{\rm(i)} For fixed $x\in\RR^d$, $\tau_x$ is a continuous linear mapping from $C^{\infty}(\RR^d)$ into itself, and satisfies $D_j(\tau_xf)=\tau_x(D_jf)$ ($j=1,\dots,d$) on $\RR^d$ for $f\in C^{\infty}(\RR^d)$;

{\rm(ii)} for fixed $x,t\in\RR^d$, the mapping $f\mapsto(\tau_xf)(t)$ defines a compactly supported distribution, whose support is contained in the ball $\{\xi\in\RR^d:\,|\xi|\le|x|+|t|\}$;

{\rm(iii)} for fixed $x\in\RR^d$, if $f\in{\mathscr S}(\RR^d)$, then $\tau_xf\in{\mathscr S}(\RR^d)$ too, and both (\ref{translation-2-0}) and (\ref{translation-2-2}) hold.

{\rm(iv)} If $f\in C^{\infty}(\RR^d)$, then the function $(x,t)\mapsto (\tau_x f)(t)$ is in $C^{\infty}(\RR^d\times\RR^d)$, and for $x,t\in\RR^d$, $(\tau_tf)(x)=(\tau_xf)(t)$.
\end{proposition}

\begin{proposition}\label{translation-2-b-1}
If $f\in C^{\infty}(\RR^d)$, then {\rm (i)} $\left[\tau_x(\sigma f)\right](t)=(\tau_{\sigma(x)}f)(\sigma(t))$ for $x,t\in\RR^d$ and $\sigma\in G$; {\rm (ii)} $\left[\tau_x(f(a\cdot))\right](t)=(\tau_{ax}f)(at)$ for $x,t\in\RR^d$ and $a\in\RR\setminus\{0\}$; {\rm (iii)} $\left[\tau_t(\tau_xf)\right](z)=\left[\tau_x(\tau_tf)\right](z)$ for $x,t,z\in\RR^d$.
\end{proposition}


In consideration of radial functions in $\RR^d$, one may learn more about the translation operator $\tau_x$. In fact, for $x,t\in\RR^d$, by \cite[Theorem 5.1]{Ro4} there exists a unique compactly supported, radial (Borel) probability measure $\rho^{\kappa}_{x,t}$ on $\RR^d$ such that for all radial $f\in C^{\infty}(\RR^d)$,
\begin{align}\label{translation-2-3}
(\tau_xf)(t)=\int_{\RR^d}f\,d\rho^{\kappa}_{x,t};
\end{align}
and the support of $\rho^{\kappa}_{x,t}$ is contained in
\begin{align}\label{translation-support-2-1}
\left\{\xi\in\RR^d:\quad \min_{\sigma\in G}|x+\sigma(t)|\le|\xi|\le\max_{\sigma\in G}|x+\sigma(t)|\right\}.
\end{align}
In particular, if $0\in{\rm supp}\,\rho^{\kappa}_{x,t}$, then the $G$-orbits of $x$ and $-t$ coincide. Since ${\rm supp}\,\rho^{\kappa}_{x,t}$ is compact, (\ref{translation-2-3}) leads to a natural extension of  the translation operator $\tau_x$ to all radial $f\in C(\RR^d)$. Furthermore, if $f\in C^{\infty}(\RR^d)$ is radial, say $f(x)=f_0(|x|)$, from the proof of \cite[Theorem 5.1]{Ro4} it follows that, for $x,t\in\RR^d$,
\begin{align}\label{translation-2-4}
(\tau_xf)(t)=\int_{\RR^d}f_0(\sqrt{|x|^2+|t|^2+2\langle t,\xi\rangle})\,d\mu^{\kappa}_{x}(\xi).
\end{align}
Again, since $\mu^{\kappa}_{x}$ is of compact support, (\ref{translation-2-4}) allows an extension of $\tau_x$ to all radial $f\in C(\RR^d)$ by a density argument.

\begin{proposition}\label{translation-2-e} {\rm (\cite[Theorems 3.7 and 3.8]{TX})}
Assume $x\in\RR^d$ and $1\le p\le2$.

{\rm(i)} The generalized translation operator $\tau_x$, initially defined on $L_{\kappa}^1(\RR^d)\cap L_{\kappa}^2(\RR^d)$, can be extended to all radial functions $f$ in $L_{\kappa}^p(\RR^d)$, $1\le p\le2$, and for these $f$, one has $\|\tau_xf\|_{L_{\kappa}^p}\le\|f\|_{L_{\kappa}^p}$;

{\rm(ii)} if $f\in L_{\kappa}^p(\RR^d)$ is radial and nonnegative, then $\tau_xf$ is nonnegative;

{\rm(iii)} if $f\in L_{\kappa}^1(\RR^d)$ is radial, then
\begin{align*}
\int_{\RR^d}(\tau_xf)(t)\,d\omega_{\kappa}(t)=\int_{\RR^d}f(t)\,d\omega_{\kappa}(t).
 \end{align*}
\end{proposition}

\subsection{The generalized spherical mean operator}

We note that it is still open whether the translation operator $\tau_x$ has a bounded extension to non-radial functions in $L_{\kappa}^1(\RR^d)$. However we may consider the generalized spherical mean operator $f\mapsto M_f$ associated to $G$ and $\kappa$, which is defined in \cite{MT}, for $x\in\RR^d$ and $r\in[0,\infty)$, by
\begin{align}\label{spherical-mean-2-1}
M_f(x,r)=d_{\kappa}\int_{\SB^{d-1}}(\tau_{x}f)(rt')W_{\kappa}(t')dt',
\end{align}
where $dt'$ denotes the area element on $\SB^{d-1}$, and $d_{\kappa}^{-1}=\int_{\SB^{d-1}}W_{\kappa}(t')dt'$. By Proposition \ref{translation-2-b}(i), $M_f$ is well defined for $f\in C^{\infty}(\RR^d)$ or $A_{\kappa}(\RR^d)$. Moreover, by \cite[Theorems 3.1, 4.1, and Corollary 5.2]{Ro4} the mapping $f\mapsto M_f$ is positivity-preserving on $C^{\infty}(\RR^d)$, and for $x\in\RR^d$ and $r\in[0,\infty)$, there exists a unique compactly supported probability (Borel) measure $\sigma^{\kappa}_{x,r}$ on $\RR^d$ such that for all $f\in C^{\infty}(\RR^d)$,
\begin{align}\label{spherical-mean-2-2}
M_f(x,r)=\int_{\RR^d}f\,d\sigma^{\kappa}_{x,r};
\end{align}
the support of $\sigma^{\kappa}_{x,r}$ is contained in
\begin{align}\label{spherical-mean-support-2-1}
\{\xi\in\RR^d:\,\,|\xi|\ge||x|-r|\}\cap\left[\cup_{\sigma\in G} \{\xi\in\RR^d:\,\,|\xi-\sigma(x)|\le r\}\right];
\end{align}
the mapping $(x,r)\mapsto\sigma_{x,r}^{\kappa}$ is
continuous with respect to the weak topology on $M^1(\RR^d)$ (the
space of probability measures);
and moreover,
\begin{align*}
\sigma_{\sigma(x),r}^{\kappa}(A)=\sigma_{x,r}^{\kappa}(\sigma^{-1}(A)), \qquad \sigma_{ax,ar}^{\kappa}(A)=\sigma_{x,r}^{\kappa}(a^{-1}A)
\end{align*}
for all $\sigma\in G$, $a>0$, and all Borel sets $A\in{\mathcal{B}}(\RR^d)$.

Since ${\rm supp}\,\rho^{\kappa}_{x,t}$ is compact, (\ref{spherical-mean-2-2}) gives a natural extension of the generalized spherical mean operator $f\mapsto M_f$ to all $f\in C(\RR^d)$.

\subsection{The $\kappa$-Poisson integral}

Associated to the reflection group $G$ on $\RR^{d}$ and the multiplicity function $\kappa$, we define $P(x)=c_{d,\kappa}(1+|x|^2)^{-|\kappa|-(d+1)/2}$, where $c_{d,\kappa}=2^{|\kappa|+d/2}\pi^{-1/2}\Gamma(|\kappa|+(d+1)/2)$. It follows that $c_{\kappa}\int_{\RR^d}P\,d\omega_{\kappa}=1$ and $\left(\SF_{\kappa}P\right)(\xi)=e^{-|\xi|}$ (cf. \cite{RV1,TX}). The associated Poisson kernel $P_y$ for $y>0$ is given by $P_y(x)=y^{-2|\kappa|-d}P(x/y)$, i.e.
\begin{align*}
P_y(x)=\frac{c_{d,\kappa}\,y}{(y^2+|x|^2)^{|\kappa|+(d+1)/2}},\qquad x\in\RR^d.
\end{align*}
For $f\in L_{\kappa}^p(\RR^d)$, $1\le p\le\infty$, the Poisson integral of $f$ is defined as
\begin{align}\label{Poisson-2-2}
(Pf)(x,y)=c_{\kappa}\int_{\RR^{d}}f(t)(\tau_{x}P_{y})(-t)\,d\omega_{\kappa}(t),\qquad (x,y)\in\RR^{d+1}_+.
\end{align}
For convenience, we call $P_y(x)$ the $\kappa$-Poisson kernel, and $(Pf)(x,y)$ the $\kappa$-Poisson integral of $f$. Note that the function $(x,y)\mapsto P_y(x)$ is $\kappa$-harmonic in $\RR_+^{d+1}$, and so is the function $(x,y)\mapsto(\tau_{x}P_{y})(-t)$ by Propositions \ref{translation-2-b}(i) and (iv). Appealing to  \cite[Proposition 5.1]{ADH1}, the following proposition is immediate.

\begin{proposition} \label{Poisson-a}
For $f\in L_{\kappa}^p(\RR^d)$, $1\le p\le\infty$, its $\kappa$-Poisson integral $(Pf)(x,y)$ is $\kappa$-harmonic in $\RR_+^{d+1}$.
\end{proposition}

The next proposition is a consequence of Proposition \ref{translation-2-b-1} and the $G$-invariance of $P_y$.

\begin{proposition} \label{Poisson-b}
If $f\in L_{\kappa}^p(\RR^d)$ ($1\le p\le\infty$) is $G$-invariant, then its $\kappa$-Poisson integral $(Pf)(x,y)$ is also $G$-invariant in $x$.
\end{proposition}

The conclusions in the following two propositions can be found in \cite{ADH1,TX}, and for $d=1$, $G=\ZZ_2$, see \cite[Proposition 3.2 and Corollary 3.9]{LL1}.

\begin{proposition} \label{Poisson-c} {\rm (\cite[Theorems 4.1 and 5.4]{TX})} For $f\in L_{\kappa}^p(\RR^d)$ ($1\le p\le\infty$) and $y>0$, $\|(Pf)(\cdot,y)\|_{L_{\kappa}^p(\RR^d)}\le\|f\|_{L_{\kappa}^p(\RR^d)}$; and for $f\in X=L_{\kappa}^p(\RR^d)$, $1\le
p<\infty$, or $C_0(\RR^d)$, $\lim_{y\rightarrow0+}\|(Pf)(\cdot,y)-f\|_{X}=0$.
\end{proposition}

\begin{proposition} \label{Poisson-d} {\rm (\cite[Corollary 5.4]{ADH1})}
Assume that $a>0$. If $f\in L^p_{\kappa}(\RR^d)$, $1\le p\le\infty$, then for almost every
$x\in\RR^d$, its $\kappa$-Poisson integral $(Pf)(t,y)$ converges
to $f(x)$ as $(t,y)\in\Gamma_a(x)$ approaching to $(x,0)$.
\end{proposition}

The next proposition is the (global) Fatou-type theorem for $\kappa$-harmonic functions.

\begin{proposition}\label{Fatou-a} {\rm (\cite[Corollary 7.4]{ADH1})}
If $u$ is $\kappa$-harmonic and bounded on $\RR^{d+1}_+$, then $u$ has a non-tangential limit at almost every point of the boundary.
\end{proposition}

\section{Some elementary propositions}

In this section we present several propositions which will be useful subsequently.
Some conclusions may be found elsewhere, but here we treat them as elementarily as possible. In particular, it will be proved that, on a general $G$-invariant domain $\Omega$ of $\RR^d$, a twice-differentiable function $f$, harmonic associated to the Dunkl operators, is infinitely differentiable and satisfies the mean value formula $M_f(x,r)=f(x)$ for all $x\in\Omega$ provided $\overline{\BB(x,r)}\subset\Omega$ with $r>0$.


%

\begin{proposition}\label{translation-2-d}
If {\rm(i)} $f\in C^{\infty}(\RR^d)$, and $g\in C(\RR^d)$ is radial and of compact support, or {\rm(ii)} $f\in C^{\infty}(\RR^d)$ with compact support, and $g\in C(\RR^d)$ is radial, then (\ref{translation-2-1}) holds.
\end{proposition}

\begin{proof}
Assume that ${\rm supp}\,g\subset\{\xi\in\RR^d:\,|\xi|\le r_0\}$. For fixed $x\in\RR^d$, we take a compactly supported $\phi\in C^{\infty}(\RR^d)$ with $\phi(\xi)=1$ for $|\xi|\le|x|+r_0$. It then follows from Proposition \ref{translation-2-b}(ii) that $[\tau_x(\phi f)](t)=(\tau_xf)(t)$ for $|t|\le r_0$, and from (\ref{translation-2-3}) and (\ref{translation-support-2-1}), $(\tau_{-x}g)(t)=0$ for $|t|>|x|+r_0$. Thus the assertion of part (i) is validated by applying Proposition \ref{translation-2-a} to $\phi f$ and $g$. A similar treatment verifies the assertion of part (ii). $\square$
\end{proof}

\begin{proposition}\label{translation-2-b-2}
If $f\in C(\RR^d)$ is radial, then $\left[\tau_x(f(a\cdot))\right](t)=(\tau_{ax}f)(at)$ for $x,t\in\RR^d$ and $a\in\RR\setminus\{0\}$.
\end{proposition}

Indeed, by Proposition \ref{translation-2-b-1} the assertion is true for $f\in{\mathscr S}(\RR^d)$. For general radial $f\in C(\RR^d)$, we take a radial and compactly supported $\phi\in C^{\infty}(\RR^d)$ with $\phi(\xi)=1$ for $|\xi|\le|a|(|x|+|t|)$, and then, choose a sequence of radial functions $\{\phi_j\}\subset{\mathscr S}(\RR^d)$ converging uniformly to $\phi f$ on $\RR^d$.
Letting $j\rightarrow\infty$ on the both sides of $\left[\tau_x(\phi_j(a\cdot))\right](t)=(\tau_{ax}\phi_j)(at)$ yields
$\left[\tau_x((\phi f)(a\cdot))\right](t)=(\tau_{ax}(\phi f))(at)$, which is identical with $\left[\tau_x(f(a\cdot))\right](t)=(\tau_{ax}f)(at)$ on account of (\ref{translation-support-2-1}).

In what follows, we set $\widetilde{\Delta}_{\kappa}=\sum_{j=1}^{d}D_{j}^{2}$.

\begin{proposition}\label{spherical-mean-2-a}
{\rm(i)} If $f\in A_{\kappa}(\RR^d)$, then
\begin{align}\label{spherical-mean-2-3}
M_f(x,r)=c_{\kappa}\int_{\RR^d}j_{|\kappa|+(d-2)/2}(r|\xi|)E_{\kappa}(i\xi,x)(\SF_{\kappa}f)(\xi)\,d\omega_{\kappa}(\xi).
 \end{align}

{\rm(ii)} If $f\in C^{\infty}(\RR^d)$, then for $x\in\RR^d$ and $r>0$,
\begin{align}\label{spherical-mean-2-4}
M_f(x,r)-f(x)=\int_0^r\frac{1}{s^{2|\kappa|+d-1}}\int_0^s \tilde{s}^{2|\kappa|+d-1}M_{\widetilde{\Delta}_{\kappa}f}(x,\tilde{s})\,d\tilde{s}.
 \end{align}
\end{proposition}

\begin{proof}
Part (i) follows from (\ref{translation-2-2}), \cite[Corollary 2.5]{Ro4} and Fubini's theorem.

As for part (ii), if $f\in{\mathscr S}(\RR^d)$, from (\ref{spherical-mean-2-3}) it is obvious that $\lim_{r\rightarrow0+}M_f(x,r)=f(x)$, and by (\ref{Dunkl-kernel-eigenfunction-1}) and Proposition \ref{transform-a}(iii), $\widetilde{\Delta}^x_{\kappa}\left[M_f(x,r)\right]=M_{\widetilde{\Delta}_{\kappa}f}(x,r)$. From \cite[(4.5)]{Ro4} we have
\begin{align*}
\frac{\partial}{\partial r}\left[r^{2|\kappa|+d-1}\frac{\partial}{\partial r}M_f(x,r)\right]=r^{2|\kappa|+d-1}\widetilde{\Delta}^x_{\kappa}\left[M_f(x,r)\right],
\end{align*}
so that (\ref{spherical-mean-2-4}) follows immediately. For general $f\in C^{\infty}(\RR^d)$ and for $x\in\RR^d$, $r>0$, we choose a compactly supported $\phi\in C^{\infty}(\RR^d)$ with $\phi(\xi)=1$ for $|\xi|\le|x|+r$. It then follows that $\phi f\in{\mathscr S}(\RR^d)$ and (\ref{spherical-mean-2-4}) holds with $\phi f$ instead of $f$. But by (\ref{spherical-mean-2-2}) and (\ref{spherical-mean-support-2-1}) $M_{\phi f}(x,r)=M_f(x,r)$ and $M_{\widetilde{\Delta}_{\kappa}(\phi f)}(x,s)=M_{\widetilde{\Delta}_{\kappa}f}(x,s)$ for $s\in[0,r)$, so that (\ref{spherical-mean-2-4}) holds for $f$ itself. $\square$
\end{proof}

\begin{proposition}\label{spherical-mean-2-b}
{\rm(i)} If $f\in C(\RR^d)$, then
\begin{align}\label{spherical-mean-2-5}
\lim_{r\rightarrow0+}M_f(x,r)=f(x)
\end{align}
for $x\in\RR^d$, and the limit holds uniformly for $x$ in a compact subset of $\RR^d$;

{\rm(ii)} if $f\in C(\RR^d)$, and $g\in C(\RR^d)$ is radial and of compact support, then for $x\in\RR^d$,
\begin{align}\label{spherical-mean-2-6}
d_{\kappa}^{-1}\int_0^{\infty}M_f(x,r)g_0(r)\,r^{2|\kappa|+d-1}dr=\int_{\RR^d}f(t)(\tau_{-x}g)(t)\,d\omega_{\kappa}(t),
 \end{align}
where $g(t)=g_0(|t|)$ for $t\in\RR^d$.
\end{proposition}

\begin{proof}
Assume that ${\rm supp}\,g\subset\{\xi\in\RR^d:\,|\xi|\le r_0\}$ and $x\in\overline{\BB(0,r_1)}$, where $r_0,r_1>0$. If $f\in C^{\infty}(\RR^d)$, (\ref{spherical-mean-2-5}) follows from (\ref{spherical-mean-2-4}), where the limit holds uniformly for all $x\in\overline{\BB(0,r_1)}$; and (\ref{spherical-mean-2-6}) follows from Proposition \ref{translation-2-d} by using the spherical coordinates to the left hand side of (\ref{translation-2-1}). For general $f\in C(\RR^d)$, we take a compactly supported $\phi\in C^{\infty}(\RR^d)$ with $\phi(\xi)=1$ for $|\xi|\le r_0+r_1$, and then, choose a sequence of functions $\{\phi_j\}\subset{\mathscr S}(\RR^d)$ converging uniformly to $\phi f$ on $\RR^d$. Note that, by (\ref{spherical-mean-2-2}) and (\ref{spherical-mean-support-2-1}) $M_{\phi f}(x,r)=M_f(x,r)$ for $r\in[0,r_0]$ and $x\in\overline{\BB(0,r_1)}$, and $M_{\phi_j}(x,r)$ converges to $M_{\phi f}(x,r)$ uniformly for $r\in[0,r_0]$ as $j$ tends to infinity; and further, by (\ref{translation-2-3}) and (\ref{translation-support-2-1}) $(\tau_{-x}g)(t)=0$ for $|t|>|x|+r_0$. Now applying (\ref{spherical-mean-2-6}) to $\phi_j$ and letting $j\rightarrow\infty$, (\ref{spherical-mean-2-6}) is proved for $f\in C(\RR^d)$. Finally for $x\in\overline{\BB(0,r_1)}$ and $r\in[0,r_0]$,
\begin{align*}
|M_f(x,r)-f(x)|\le 2\|\phi f-\phi_j\|_{C_0(\RR^d)}+|M_{\phi_j}(x,r)-\phi_j(x)|,
\end{align*}
and then, letting $r\rightarrow0+$ and $j\rightarrow\infty$ successively proves (\ref{spherical-mean-2-5}) for $f\in C(\RR^d)$ and the limit in (\ref{spherical-mean-2-5}) holds uniformly for all $x\in\overline{\BB(0,r_1)}$. $\square$
\end{proof}


We are now in a position to prove a mean-value property of functions satisfying $\widetilde{\Delta}_{\kappa}f=0$ in a general case. Such a problem was earlier considered in \cite{GR1,MT}.

\begin{proposition}\label{harmonic-2-a}
Let $\Omega$ be a $G$-invariant domain of $\RR^d$. If $f\in C^2(\Omega)$ satisfies $\widetilde{\Delta}_{\kappa}f=0$ in $\Omega$, then $f\in C^{\infty}(\Omega)$, and for $x\in\Omega$,
\begin{align}\label{spherical-mean-2-7}
M_f(x,r)=f(x)
\end{align}
provided $\overline{\BB(x,r)}\subset\Omega$ for $r>0$.
\end{proposition}

\begin{proof}
Let $\phi\in C^{\infty}(\RR^d)$ be a radial function satisfying $\int_{\RR^d}\phi\,d\omega_{\kappa}=1$ and ${\rm supp}\,\phi\subseteq\BB(0,1)$, and write $\phi(t)=\tilde{\phi}(|t|)$ for $t\in\RR^d$.
For given $\overline{\BB(x,r)}\subset\Omega$ with $r>0$, choose $r''>r'>r$ such that $\overline{\BB(x,r'')}\subset\Omega$. By Urysohn's lemma, there exists a function $\psi\in C^{\infty}(\RR^d)$ such that $\psi(t)=1$ for $t\in\cup_{\sigma\in G}\overline{\BB(\sigma(x),r')}$ and $\psi(t)=0$ for $t\notin\cup_{\sigma\in G}\BB(\sigma(x),r'')$. Thus $\psi f$ has a natural $C^2$ extension to $\RR^d$ with compact support.

For $\epsilon>0$, define
\begin{align}\label{auxiliary-function-2-0}
f_{\epsilon}(\tilde{x})=\int_{\RR^d}\psi(t)f(t)(\tau_{-\tilde{x}}\phi_{\epsilon})(t)\,d\omega_{\kappa}(t),
\end{align}
where $\phi_{\epsilon}(t)=\epsilon^{-2|\kappa|-d}\phi(\epsilon^{-1}t)$. Obviously $f_{\epsilon}\in C^{\infty}(\RR^d)$.

Since by (\ref{translation-2-2}), $\widetilde{\Delta}^x_{\kappa}\left[(\tau_{-x}\phi_{\epsilon})(t)\right]=\widetilde{\Delta}^t_{\kappa}\left[(\tau_{-x}\phi_{\epsilon})(t)\right]$, the Green formula in \cite[Theorem 4.11]{MT} gives
\begin{align}\label{auxiliary-function-2-1}
\left(\widetilde{\Delta}_{\kappa}f_{\epsilon}\right)(\tilde{x})&=\int_{\RR^d}\psi(t)f(t)\,\widetilde{\Delta}^t_{\kappa}\left[(\tau_{-\tilde{x}}\phi_{\epsilon})(t)\right]\,d\omega_{\kappa}(t)\nonumber\\
&=\int_{\RR^d}(\tau_{-\tilde{x}}\phi_{\epsilon})(t)\left[\widetilde{\Delta}_{\kappa}(\psi f)\right](t)\,d\omega_{\kappa}(t).
\end{align}
We assume that $\tilde{x}\in\cup_{\sigma\in G}\overline{\BB(\sigma(x),r)}$, says $\tilde{x}\in\overline{\BB(\sigma_1(x),r)}$ for some $\sigma_1\in G$. For $0<\epsilon<r'-r$, ${\rm supp}\,\phi_{\epsilon}\subseteq\BB(0,r'-r)$, and hence, (\ref{translation-support-2-1}) implies
$$
{\rm supp}\,\phi_{\epsilon}\cap{\rm supp}\,\rho^{\kappa}_{-\tilde{x},t}=\emptyset\qquad\hbox{for}\quad t\notin\cup_{\sigma\in G}\BB(\sigma(\tilde{x}),r'-r),
$$
so that $(\tau_{-\tilde{x}}\phi_{\epsilon})(t)=0$ by (\ref{translation-2-3}). But for $t\in\cup_{\sigma\in G}\BB(\sigma(\tilde{x}),r'-r)$, letting $|t-\sigma_2(\tilde{x})|<r'-r$ for some $\sigma_2\in G$, it follows that $|t-\sigma_2(\sigma_1(x))|\le|t-\sigma_2(\tilde{x})|+|\sigma_2(\tilde{x}-\sigma_1(x))|<r'$, i.e., $t\in\cup_{\sigma\in G}\BB(\sigma(x),r')$, and hence
$[\widetilde{\Delta}_{\kappa}(\psi f)](t)=(\widetilde{\Delta}_{\kappa}f)(t)=0$ by the assumption. Thus (\ref{auxiliary-function-2-1}) implies  $(\widetilde{\Delta}_{\kappa}f_{\epsilon})(\tilde{x})=0$ for $\tilde{x}\in\cup_{\sigma\in G}\overline{\BB(\sigma(x),r)}$, and consequently, $M_{\widetilde{\Delta}_{\kappa}f_{\epsilon}}(x,s)=0$ for $s\in(0,r)$ by (\ref{spherical-mean-2-2}) and (\ref{spherical-mean-support-2-1}). Now applying (\ref{spherical-mean-2-4}) to $f_{\epsilon}$ we get
\begin{align}\label{spherical-mean-2-8}
M_{f_{\epsilon}}(x,r)=f_{\epsilon}(x).
 \end{align}
Also, for $\tilde{x}\in\cup_{\sigma\in G}\overline{\BB(\sigma(x),r)}$, by Proposition \ref{spherical-mean-2-b} and (\ref{auxiliary-function-2-0}) we have
\begin{align}\label{spherical-mean-2-9}
f_{\epsilon}(\tilde{x})=d_{\kappa}^{-1}\int_0^{\infty}M_{\psi f}(\tilde{x},\epsilon s)\tilde{\phi}(s)\,s^{2|\kappa|+d-1}ds,
\end{align}
and since ${\rm supp}\,\tilde{\phi}\subseteq[0,1]$, Proposition \ref{spherical-mean-2-b}(i) implies that $f_{\epsilon}(\tilde{x})$ tends to $\psi(\tilde{x})f(\tilde{x})=f(\tilde{x})$ as $\epsilon\rightarrow0+$ uniformly for $\tilde{x}\in\cup_{\sigma\in G}\overline{\BB(\sigma(x),r)}$. Then by (\ref{spherical-mean-2-2}) and (\ref{spherical-mean-support-2-1}), (\ref{spherical-mean-2-7}) is obtained by taking $\epsilon\rightarrow0+$ in (\ref{spherical-mean-2-8}).

Finally we fix a number $\epsilon\in(0,r'-r)$. It has been shown that $f_{\epsilon}\in C^{\infty}(\BB(x,r))$, and for $\tilde{x}\in\BB(x,r)$ and $s\in[0,1]$, $\cup_{\sigma\in G}\BB(\sigma(\tilde{x}),\epsilon s)\subset\cup_{\sigma\in G}\BB(\sigma(x),r')$, so that $M_{\psi f}(\tilde{x},\epsilon s)=M_{f}(\tilde{x},\epsilon s)$. Thus by (\ref{spherical-mean-2-7}) and (\ref{spherical-mean-2-9}), we have $f(\tilde{x})=f_{\epsilon}(\tilde{x})$ for $\tilde{x}\in\BB(x,r)$, and hence $f\in C^{\infty}(\BB(x,r))$.
The proof of the proposition is completed. $\square$
\end{proof}

Notice that, the final part of the proof of the above proposition implies that a function $f\in C(\Omega)$ satisfying the mean value property (\ref{spherical-mean-2-7}) must be infinitely differentiable. Further, if a function $f\in C^2(\Omega)$ satisfying (\ref{spherical-mean-2-7}), Proposition \ref{spherical-mean-2-a}(ii) gives that $M_{\widetilde{\Delta}_{\kappa}f}(x,r)\equiv0$ for $r>0$ with $\overline{\BB(x,r)}\subset\Omega$, and so $\widetilde{\Delta}_{\kappa}f(x)=0$ by Proposition \ref{spherical-mean-2-b}(ii).

\begin{corollary}\label{harmonic-2-a-1}
Assume that $\Omega$ is a $G$-invariant domain of $\RR^d$ and $f\in C^2(\Omega)$. Then $f$ satisfies $\widetilde{\Delta}_{\kappa}f=0$ in $\Omega$ if and only if for all $x\in\Omega$, $M_f(x,r)=f(x)$
provided $\overline{\BB(x,r)}\subset\Omega$ for $r>0$. Moreover, in both cases, $f\in C^{\infty}(\Omega)$.
\end{corollary}

Such a characterization for $\Omega=\RR^d$ and $f\in C^{\infty}(\RR^d)$ was proved in \cite{MT}. For a general $G$-invariant domain $\Omega$ and $f\in C^2(\Omega)$, a different characterization, in terms of a generalized volume mean value property, was given in \cite{GR1}; moreover, by \cite[Remark 3.2]{GR1}, the method there also yields a characterization by means of the generalized spherical mean value property $M_f(x,\rho)=f(x)$ for $0<\rho\le r/3$ provided $\overline{\BB(x,r)}\subset\Omega$ with $r>0$.

In what follows we consider the reflection group $\tilde{G}=G\otimes\ZZ_2$ on $\RR^{d+1}$ and the multiplicity function $\tilde{\kappa}=(\kappa,0)$. In this case, since the multiplicity value associated to $\ZZ_2$ is zero, the $G\otimes\ZZ_2$-invariance of the domain $\Omega$ in \cite[Theorem 4.11]{MT} and Proposition \ref{harmonic-2-a} can be relaxed to the $G$-invariance. We reformulate the Green formula in \cite[Theorem 4.11]{MT} as follows.

\begin{proposition}\label{Green-formula-2-a}
Let $\Omega$ be a bounded, regular, and $G$-invariant domain of $\RR^{d+1}$. Then for $u,v\in C^2(\overline{\Omega})$,
\begin{align}\label{Green-formula-2-1}
\iint_{\Omega}(v\Delta_{\kappa}u-u\Delta_{\kappa}v)\,d\omega_{\kappa}(x)dy
=\int_{\partial\Omega}\left(v\partial_{\mathbf{n}}u
-u\partial_{\mathbf{n}}v\right)W_{\kappa}(x)d\sigma(x,y),
\end{align}
where $\partial_{\mathbf{n}}$ denotes the directional derivative of the outward normal, and $d\sigma(x,y)$ the area element on $\partial\Omega$.
\end{proposition}

We have a variant of (\ref{Green-formula-2-1}) with the directional Dunkl operator $D_{\mathbf{n}}$ instead of the usual directional derivative $\partial_{\mathbf{n}}$. For $\alpha\in R$, write $\tilde{\alpha}=(\alpha,0)$ as a vector in $\RR^{d+1}$; and for a unit vector $\mathbf{n}$ in $\RR^{d+1}$, the directional Dunkl operator $D_{\mathbf{n}}$ is given by $D_{\mathbf{n}}=\langle\nabla_{\kappa},\mathbf{n}\rangle$, where $\nabla_{\kappa}=(D_1,\dots,D_d,\partial_y)$ is the $\kappa$-gradient. From (\ref{Dunkl-operator-1}),
\begin{eqnarray*}
D_{\mathbf{n}}=\partial_{\mathbf{n}}+\sum_{\alpha\in R_+}\kappa(\alpha)\langle\tilde{\alpha},{\mathbf{n}}\rangle\frac{1-\sigma_{\alpha}}{\langle\alpha,x\rangle},
 \end{eqnarray*}
and for $\alpha\in R_+$, since $\partial\Omega$ is $G$-invariant, the change of variables $x\rightarrow\sigma_{\alpha}(x)$ shows that
\begin{align*}
\int_{\partial\Omega}\langle\tilde{\alpha},{\mathbf{n}}\rangle\left(v\frac{u-\sigma_{\alpha}u}{\langle\alpha,x\rangle}
-u\frac{v-\sigma_{\alpha}v}{\langle\alpha,x\rangle}\right)W_{\kappa}(x)d\sigma(x,y)=0.
\end{align*}
Summing the above expression with the multiple $\kappa(\alpha)$ over $\alpha\in R_+$ and inserting the result into the right hand side of (\ref{Green-formula-2-1}), we obtain the following form of the Green formula.
\begin{proposition}\label{Green-formula-2-b}
Let $\Omega$ be a bounded, regular, and $G$-invariant domain of $\RR^{d+1}$. Then for $u,v\in C^2(\overline{\Omega})$,
\begin{align*}
\iint_{\Omega}(v\Delta_{\kappa}u-u\Delta_{\kappa}v)\,d\omega_{\kappa}(x)dy
=\int_{\partial\Omega}\left(vD_{\mathbf{n}}u
-uD_{\mathbf{n}}v\right)W_{\kappa}(x)d\sigma(x,y),
\end{align*}
where ${\mathbf{n}}$ denotes the outward normal of $\partial\Omega$.
\end{proposition}

The next proposition is a consequence of Proposition \ref{harmonic-2-a} with $\Omega=\RR^{d+1}_+=\RR^d\times(0,\infty)$.

\begin{proposition}\label{harmonic-2-b}
If the $C^2$ function $u$ on the upper-space $\RR^{d+1}_+$ is $\kappa$-harmonic, then $u\in C^{\infty}(\RR^{d+1}_+)$.
\end{proposition}

\section{The proof of Theorem \ref{thm-1}}

That part (i) implies part (ii) in Theorem \ref{thm-1} is obvious; what is left is to show that part (ii) implies part (i) in Theorem \ref{thm-1}, which is restated as follows.


\begin{theorem}\label{bounded-convergence}
Suppose that $E$ is a $G$-invariant measurable subset of $\partial\RR^{d+1}_+=\RR^d$ and $u$ is $\kappa$-harmonic in $\RR^{d+1}_+$. If $u$ is non-tangentially bounded at $(x,0)$ for every $x\in E$, then there exists a subset $E_0$ of full measure of $E$ such that for every $a>0$ and each $x\in E_0$, $u(t,y)$ has a finite limit as $(t,y)\in\Gamma_a(x)$ approaches to $(x,0)$.
\end{theorem}

For a subset $E$ of $\partial\RR^{d+1}_+=\RR^d$ and for fixed $a,h>0$, we always use the notation
\begin{eqnarray*}
\Omega^E(a,h):=\bigcup_{x^{0}\in E}\Gamma^{h}_{a}(x^{0}).
\end{eqnarray*}

We shall need a series of lemmas.

\begin{lemma}\label{fatou-a} {\rm (\cite[p. 201, Lemma]{St2})}
Let $u$ be a continuous function in $\RR^{d+1}_+$, and $E$ a measurable set of $\partial\RR^{d+1}_+=\RR^d$ with $0<|E|_0<\infty$, where $|\cdot|_0$ denotes the Lebesgue measure on $\RR^d$. If $u$ is non-tangentially bounded at $(x,0)$ for every $x\in E$, then for any $\epsilon>0$, there exists a compact set $E_{\epsilon}$ satisfying

{\rm(i)}\ \ $E_{\epsilon}\subset E$, $|E-E_{\epsilon}|_{0}<\epsilon$;

{\rm(ii)}\ for any $a>0$ and $h>0$, there is a contant $c_{a,h,\epsilon}>0$, so that $|u(x,y)|\leq c_{a,h,\epsilon}$, $(x,y)\in\Omega^{E_{\epsilon}}(a,h)$.
\end{lemma}

In \cite{ADH1, DH1}, useful lower and upper bounds of the translated $\kappa$-Poisson kernel $(\tau_x P_y)(-t)$ are obtained, namely, there are two constants $c_1,c_2>0$, independent of $y>0$ and $x,t\in\RR^d$, such that
\begin{align}\label{Poisson-ker-estimate-1}
\frac{c_1}{V(x,t,y+|x-t|)}\frac{y}{y+|x-t|}\leq(\tau_x P_y)(-t)\leq
\frac{c_2y}{V(x,t,y+d(x,t))}\frac{y+d(x,t)}{y^2+|x-t|^2},
\end{align}
where $V(x,t,r)=\max\{\left|\BB(x,r)\right|_{\kappa},\left|\BB(t,r)\right|_{\kappa}\}$, and $d(x,t)=\min_{\sigma\in G}|x-\sigma(t)|$. We note that the bounds given in (\ref{Poisson-ker-estimate-1}) have simpler but equivalent forms, which are stated as follows.

\begin{lemma}\label{Poisson-ker-estimate-a}
There exist two constants $c_1,c_2>0$, independent of $y>0$ and $x,t\in\RR^d$, such that
\begin{align}\label{Poisson-ker-estimate-2}
\frac{c_1}{\left|\BB(x,y+|x-t|)\right|_{\kappa}}\frac{y}{y+|x-t|}
\leq(\tau_x P_y)(-t)\leq
\frac{c_2y}{\left|\BB(x,y+d(x,t))\right|_{\kappa}}\frac{y+d(x,t)}{y^2+|x-t|^2}.
\end{align}
\end{lemma}

To verify (\ref{Poisson-ker-estimate-2}), it suffices to show
\begin{align}
\left|\BB(x,y+|x-t|)\right|_{\kappa}&\asymp\left|\BB(t,y+|x-t|)\right|_{\kappa},\label{ball-estimate-1}\\
\left|\BB(x,y+d(x,t))\right|_{\kappa}&\asymp\left|\BB(t,y+d(x,t))\right|_{\kappa}.\label{ball-estimate-2}
\end{align}
The equation (\ref{ball-estimate-1}) is obvious since
 \begin{align}\label{ball-estimate-3}
\left|\BB(x,r)\right|_{\kappa}\asymp r^d\prod_{\alpha\in R}(|\langle\alpha,x\rangle|+r)^{\kappa(\alpha)}.
 \end{align}
As for (\ref{ball-estimate-2}), suppose $\tilde{\sigma}\in G$ such that $|x-\tilde{\sigma}(t)|=d(x,t)$. It follows that
 \begin{align}\label{ball-estimate-4}
\prod_{\alpha\in R}(|\langle\alpha,t\rangle|+y+d(x,t))^{k(\alpha)}\le\prod_{\alpha\in R}(|\langle\tilde{\sigma}(\alpha),x\rangle|+y+(1+\sqrt{2})|x-\tilde{\sigma}(t)|)^{k(\alpha)},
 \end{align}
and since $k(\alpha)=k(\tilde{\sigma}(\alpha))$, $\tilde{\sigma}(R)=R$, (\ref{ball-estimate-4}) implies that $\left|\BB(t,y+d(x,t))\right|_{\kappa}\le c\left|\BB(x,y+d(x,t))\right|_{\kappa}$ with some fixed $c>0$. By symmetry of $x$ and $t$, (\ref{ball-estimate-2}) is verified.


The next lemma is a maximum principle which is a little stronger than that in \cite[Theorem 4.2]{Ro2}.

\begin{lemma} \label{max-principle}{\rm (maximum principle)} Let $\Omega$ be a $G$-invariant bounded domain in $\RR^{d+1}$, and let $u(x,y)$ be continuous on
$\overline{\Omega}$. Assume that $u$ is of class $C^2$ in the region
where $u>0$ and satisfies $\Delta_{\kappa}u\geq 0$ there. If
$u|_{\partial \Omega} \leq 0$, then $u\le 0$ on the whole $\Omega $.
\end{lemma}

To show this, set
$u_{\epsilon}=u+\epsilon(|x|^2+y^2)-\epsilon\max_{\overline{\Omega}}(|x|^2+y^2)$
for small $\epsilon>0$. It is clear that
$u_{\epsilon}|_{\partial\Omega}\le 0$, and in the region where
$u_{\epsilon}>0$,
$\Delta_{\kappa}u_{\epsilon}=\Delta_{\kappa}u+2\epsilon(2|\kappa|+d+1)>0$. We claim that $u_{\epsilon}\leq 0$ on whole
$\Omega$, so that $u\leq 0$. Otherwise, $u_{\epsilon}$ would attain
its absolute (positive) maximum at some point $(x_0,y_0)$ in the interior of
$\Omega$; but by \cite[Lemma 4.1]{Ro2}, $\Delta_{\kappa}u_{\epsilon}(x_0,y_0)\le0$. This leads to a contradiction.

For a positive integer $m$, consider an $m$-tuple $F=(u_1,u_2,\dots,u_m)$ of real-valued functions defined on $\RR^{d+1}_+$. We shall use the the following notations: $|F|=(\sum_{j=1}^mu_j^2)^{1/2}$, $\sigma F=(\sigma u_1,\sigma u_2,\dots,\sigma u_m)$, $\langle F,\sigma F\rangle=\sum_{j=1}^mu_j\sigma u_j$, and $\langle F,\partial_{x_i}F\rangle=\sum_{j=1}^mu_j\partial_{x_i}u_j$ for $0\le i\le d$ with $x_0=y$. The following proposition is a crucial element in the proof of Theorem \ref{bounded-convergence}.

\begin{proposition}\label{subharmonicity-3-b}
If $u_1,u_2,\dots,u_{m}$ are all $\kappa$-harmonic in $\RR^{d+1}_+$, then, with the notations as above, $\Delta_{\kappa}|F|\geq 0$ in the region
where $|F|>0$.
\end{proposition}

\begin{proof}
In the region where $|F|>0$, direct calculations show that $\partial_{x_i}|F|=|F|^{-1}\langle F,\partial_{x_i}F\rangle$ and
$$
\partial_{x_i}^2|F|=|F|^{-1}|\partial_{x_i}F|^2-|F|^{-3}\langle F,\partial_{x_i}F\rangle+|F|^{-1}\langle F,\partial_{x_i}^2F\rangle.
$$
Since every $u_j$ is $\kappa$-harmonic, from (\ref{Laplace-2-1}) and (\ref{Laplace-2-2}) it follows that
\begin{align*}
\Delta|F|=\frac{1}{|F|^{3}}\sum_{i=0}^d\left(|F|^2|\partial_{x_i}F|^2-\langle F,\partial_{x_i}F\rangle^2\right)
-\frac{2}{|F|}\sum_{\alpha\in R_+}\kappa(\alpha)\langle F,\delta_{\alpha} F\rangle
\end{align*}
and
\begin{align*}
2\sum_{\alpha\in R_+}\kappa(\alpha)\delta_{\alpha}|F|=\frac{2}{|F|}\sum_{\alpha\in R_+}\kappa(\alpha)\langle F,\delta_{\alpha} F\rangle
+\frac{2}{|F|}\sum_{\alpha\in R_+}\kappa(\alpha)\frac{|F||\sigma_{\alpha}F|-\langle F,\sigma_{\alpha} F\rangle}{\langle\alpha,x\rangle^2}.
\end{align*}
Summing them gives
\begin{align*}
\Delta_{\kappa}|F|=\frac{1}{|F|^{3}}\sum_{i=0}^d\left(|F|^2|\partial_{x_i}F|^2-\langle F,\partial_{x_i}F\rangle^2\right)
+\frac{2}{|F|}\sum_{\alpha\in R_+}\kappa(\alpha)\frac{|F||\sigma_{\alpha}F|-\langle F,\sigma_{\alpha} F\rangle}{\langle\alpha,x\rangle^2}.
\end{align*}
Obviously both terms on the right hand above are nonnegative. $\square$
\end{proof}

For $m=(d+1)|G|$ (where $|G|$ is the order of $G$) and $u_1,u_2,\dots,u_{d+1}$ satisfying the generalized Cauchy-Riemann equations in terms of the Dunkl operators, a sharper conclusion was obtained in \cite{ADH1}, that is, there exists some $q\in(0,1)$ such that $\Delta_{\kappa}\left(\sum_{\sigma\in G}\sum_{j=1}^{d+1}|\sigma u_j|^2\right)^{q/2}\geq 0$ in the region where $\sum_{j=1}^{d+1}u_j^2\neq0$. The difference between this and Proposition \ref{subharmonicity-3-b} is that, we need not restrict the number of functions, while the assumption that the Cauchy-Riemann equations are satisfied is, in fact, replaced by the fact that the functions are $\kappa$-harmonic.

%

We will adopt the strategy of the proof of \cite[p. 201, Theorem 3]{St2} in proving Theorem \ref{bounded-convergence}. For convenience we extract the following two lemmas. In the proof of the later one, we need a $G$-invariant auxiliary function, by which an associated domain defined is $G$-invariant. This makes Lemma \ref{max-principle} and Proposition \ref{subharmonicity-3-b} applicable.

\begin{lemma}\label{fatou-b}
Let $E$ be a $G$-invariant measurable subset of $\partial\RR^{d+1}_+=\RR^d$, and for $a>0$, $\Omega=\Omega^E(a,1)$. Then there exists a nonnegative $\kappa$-harmonic function $H$ on $\RR^{d+1}_+$ satisfying

{\rm(i)} \ \ $H(x,y)$ is $G$-invariant in $x$ and $H(x,y)\ge2$ for $(x,y)\in\RR^{d+1}_+\bigcap\partial\Omega$;

{\rm(ii)} \ $H$ has non-tangential limit $0$ at $(x,0)$ for almost every $x\in E$.
\end{lemma}

\begin{proof}
We define
\begin{eqnarray*}
H_0(x,y)=\left(P\chi_{E^c}\right)(x,y)+y,
\end{eqnarray*}
where $\chi_{E^c}$ is the characteristic function of the complement $E^c$ of $E$. It is obvious that $H_0$ is nonnegative and $\kappa$-harmonic in $\RR^{d+1}_+$, and also $G$-invariant in $x$ by Proposition \ref{Poisson-b}. From Proposition \ref{Poisson-d}, $H$ has non-tangential limit $0$ at $(x,0)$ for almost every $x\in E$.

Now we prove that $H_0$ has a positive lower bound on $\RR^{d+1}_+\bigcap\partial\Omega$. For $(x,y)\in\partial\Omega$ with $y=1$, $H(x,y)\ge1$. If $(x,y)\in\partial\Omega$ with $0<y<1$, then $\{t:\,|t-x|<ay\}\subset E^c$; this is because, $t'\in E$ with $|t'-x|<ay$ implies that $(x,y)\in \Gamma^{1}_{a}(t')\subset\Omega$. Thus, by (\ref{Poisson-2-2}) we have
\begin{eqnarray*}
H_0(x,y)\ge c_{\kappa}\int_{|t-x|<ay}(\tau_{x}P_{y})(-t)\,d\omega_{\kappa}(t),
\end{eqnarray*}
and using the first inequality in (\ref{Poisson-ker-estimate-2}),
\begin{eqnarray*}
H_0(x,y)\ge c\int_{|t-x|<ay} \frac{d\omega_{\kappa}(t)}{\left|\BB(x,y+|x-t|)\right|_{\kappa}}
\ge\frac{c\left|\BB(x,ay)\right|_{\kappa}}{\left|\BB(x,(a+1)y)\right|_{\kappa}}.
\end{eqnarray*}
This shows that $H_0$ has a lower bound $c_0>0$ on $\RR^{d+1}_+\bigcap\partial\Omega$.
Finally the function $H=2H_0/c_0$ is desired. $\square$
\end{proof}

\begin{lemma}\label{fatou-c} Suppose that $E$ is a $G$-invariant compact subset of $\partial\RR^{d+1}_+=\RR^d$ and the function $u$ is $\kappa$-harmonic in $\RR^{d+1}_+$. If
\begin{eqnarray}\label{fatou-c-cond}
|u(x,y)|\le1,\qquad (x,y)\in\Omega,
\end{eqnarray}
where $\Omega=\Omega^E(a,2)$ for some $a>0$,
then for almost every $x\in E$, $u(t,y)$ has a finite limit as $(t,y)\in\Gamma_a(x)$ approaching to $(x,0)$.
\end{lemma}

\begin{proof} We first note that $\Omega$ is an open, $G$-invariant, bounded domain in $\RR^{d+1}_+$.
Choose a sequence $\{y_k\}_{k=1}^{\infty}\subset(0,1)$ such that $y_k\rightarrow0$, and for $x\in\RR^d$, define
$$
\varphi_{k}(x)=\left\{\begin{aligned}
                 &u(x,y_k),\quad  \text{if}\ (x,y_k)\in \Omega,
                 \\&0, \qquad \qquad \text{otherwise}. \\
               \end{aligned}\right.
$$
Obviously $|\varphi_{k}|\le1$ on $\RR^d$ for all $k\ge1$, and thus, one can find a function $\varphi$ with $|\varphi|\leq1$ on $\RR^d$ and a subsequence $\{\varphi_{k_j}\}$ so that
$\{\varphi_{k_j}\}$ converges weakly$^*$ to $\varphi$. In particular, for their $\kappa$-Poisson integrals, we have
$\lim_{j\rightarrow\infty}(P\varphi_{k_j})(x,y)=(P\varphi)(x,y)$, $(x,y)\in\RR^{d+1}_+$.

If we write $\psi_{k}(x,y)=u(x,y+y_k)-(P\varphi_{k})(x,y)$, then
\begin{eqnarray}\label{psi-1}
\psi(x,y):=\lim_{j\rightarrow\infty}\psi_{k_j}(x,y)=u(x,y)-(P\varphi)(x,y),\qquad (x,y)\in\RR^{d+1}_+.
\end{eqnarray}
Further we claim that, for each $k$,
\begin{eqnarray}\label{psi-H-1}
|\psi_k(x,y)|\le \sqrt{|G|}H(x,y),\qquad (x,y)\in\Omega_1:=\Omega^E(a,1),
\end{eqnarray}
where $H$ is the function given in Lemma \ref{fatou-b}. In fact, we shall prove the following stronger version of (\ref{psi-H-1})
\begin{eqnarray}\label{psi-H-1-1}
\Psi_k(x,y)\le \sqrt{|G|}H(x,y),\qquad (x,y)\in\Omega_1,
\end{eqnarray}
where $\Psi_k(x,y)=\left(\sum_{\sigma\in G}|\psi_k(\sigma(x),y)|^2\right)^{1/2}$.

If there were some $(x^0,y^0)\in\Omega_1$ so that $\Psi_k(x^0,y^0)>\sqrt{|G|}H(x^0,y^0)$, we consider the set
$$
\Omega_0=\{(x,y)\in\Omega_1:\, \Psi_k(x,y)-\sqrt{|G|}H(x,y)>\epsilon_0\},
$$
where $\epsilon_0=\frac{1}{2}(\Psi_k(x^0,y^0)-\sqrt{|G|}H(x^0,y^0))$. It follows that $\Omega_0$ is non-empty, open, and also $G$-invariant, since both $\Psi_k$ and $H$ are $G$-invariant. We assert that
\begin{align*}
\partial\Omega_0\bigcap\partial\Omega_1\neq\emptyset.
\end{align*}
Otherwise, $\overline{\Omega_0}\subset\Omega_1$ and $\tilde{\Psi}_k=0$ on $\partial\Omega_0$, where $\tilde{\Psi}_k=\Psi_k-\sqrt{|G|}H-\epsilon_0$. It follows that $\tilde{\Psi}_k\in C^2(\overline{\Omega_0})$ and $\Delta_{\kappa}\tilde{\Psi}_k\geq 0$ in $\Omega_0$ by Proposition \ref{subharmonicity-3-b} and Lemma \ref{fatou-b}. Lemma \ref{max-principle} (the maximum principle) yields that $\tilde{\Psi}_k\le0$ on the whole $\Omega_0$, which contradicts the definition of $\Omega_0$.

For $(x^*,y^*)\in \partial\Omega_0\bigcap\partial\Omega_1$, there exists a sequence of points $\left\{(\hat{x}_\ell,\hat{y}_{\ell})\right\}\subset\Omega_0$ converging to $(x^*,y^*)$, so that
\begin{eqnarray}\label{psi-H-2}
\sqrt{|G|}H(\hat{x}_\ell,\hat{y}_{\ell})+\epsilon_0<\Psi_k(\hat{x}_\ell,\hat{y}_{\ell}),\qquad \ell=1,2,\dots.
\end{eqnarray}
If $y^*>0$, letting $\ell\rightarrow\infty$ gives $\sqrt{|G|}H(x^*,y^*)+\epsilon_0\le\Psi_k(x^*,y^*)$; but
by Proposition \ref{Poisson-c} and (\ref{fatou-c-cond}), $|\psi_k(\sigma(x^*),y^*)|\le|u(\sigma(x^*),y^*+y_k)|+|(P\varphi_{k})(\sigma(x^*),y^*)|\le2$, which implies $\Psi_k(x^*,y^*)\le2\sqrt{|G|}$ and leads to a contradiction with the fact $H(x^*,y^*)\ge2$ by Lemma \ref{fatou-b}. Hence $y^*=0$, i.e. $\lim_{\ell\rightarrow\infty}(\hat{x}_\ell,\hat{y}_{\ell})=(x^*,0)$, and further $\sigma(x^*)\in E$ for all $\sigma\in G$.

Since $\varphi_{k}(x)=u(x,y_k)$ for $|x-x^*|<ay_k$, $\varphi_k$ is continuous at $x^*$. For given $\epsilon>0$, take $\delta_1>0$ so that $|\varphi_{k}(t)-\varphi_{k}(x^*)|<\epsilon$ whenever $|t-x^*|\le\delta_1$, and then, choose $N_1\in\NN$ such that $|\hat{x}_\ell-x^*|\le\delta_1/2$ for $\ell\ge N_1$.
Now from (\ref{Poisson-2-2}), for $\ell\ge N_1$ we have
\begin{align*}
\left|(P\varphi_{k})(\hat{x}_\ell,\hat{y}_{\ell})-\varphi_{k}(x^*)\right|
&\le c_{\kappa}\int_{\RR^{d}}\left|\varphi_{k}(t)-\varphi_{k}(x^*)\right|
\left(\tau_{\hat{x}_\ell}P_{\hat{y}_{\ell}}\right)(-t)\,d\omega_{\kappa}(t)\\
&\le \epsilon+2c_{\kappa} \int_{|t-\hat{x}_\ell|>\delta_1/2}
\left(\tau_{\hat{x}_\ell}P_{\hat{y}_{\ell}}\right)(-t)\,d\omega_{\kappa}(t).
\end{align*}
By \cite[Proposition 5.2]{ADH1}, there exists $\delta_2>0$ such that for $0<\hat{y}_{\ell}<\delta_2$, the last integral above is controlled by $\epsilon/2c_{\kappa}$. Now choose $N\ge N_1$ so that $0<\hat{y}_{\ell}<\delta_2$ for $\ell>N$. It follows that for $\ell>N$, $\left|(P\varphi_{k})(\hat{x}_\ell,\hat{y}_{\ell})-\varphi_{k}(x^*)\right|\le2\epsilon$. This shows that $(P\varphi_{k})(\hat{x}_\ell,\hat{y}_{\ell})$ tends to $\varphi_{k}(x^*)$ as $\ell\rightarrow\infty$. (Note that we could not refer to \cite[Corollary 5.3]{ADH1} directly because the function $\varphi_{k}$ is not continuous on $\RR^d$.)
Therefore
$\lim_{\ell\rightarrow\infty}\psi_k(\hat{x}_\ell,\hat{y}_{\ell})=u(x^*,y_k)-\varphi_{k}(x^*)=0$, and similarly $\lim_{\ell\rightarrow\infty}\psi_k(\sigma(\hat{x}_\ell),\hat{y}_{\ell})=0$ for all $\sigma\in G$. Thus $\lim_{\ell\rightarrow\infty}\Psi_k(\hat{x}_\ell,\hat{y}_{\ell})=0$, and from (\ref{psi-H-2}) we have
$$
\limsup_{\ell\rightarrow\infty}\sqrt{|G|}H(\hat{x}_\ell,\hat{y}_{\ell})\le-\epsilon_0,
$$
which contradicts with nonnegativity of $H$ by Lemma \ref{fatou-b}. The claim (\ref{psi-H-1-1}), and so (\ref{psi-H-1}), are proved.

Finally, taking $k=k_j$ in (\ref{psi-H-1}) and letting $j\rightarrow\infty$ yields $|\psi(x,y)|\le\sqrt{|G|}H(x,y)$ for $(x,y)\in\Omega_1$, and by Lemma \ref{fatou-b}, for almost every $x\in E$, $\psi(t,y)$ tends to zero as $(t,y)\in\Gamma_a(x)$ approaching to $(x,0)$. Further by Proposition \ref{Poisson-d}, $(P\varphi)(x,y)$ has a finite non-tangential limit for almost every $x\in E$, and hence, from (\ref{psi-1}), $u(x,y)=\psi(x,y)+(P\varphi)(x,y)$ has the desired assertion in the lemma. $\square$
\end{proof}

Now we turn to the proof of Theorem \ref{bounded-convergence}.


We assume that the set $E$ is bounded, without loss of generality. By Lemma \ref{fatou-a}, for each $k\in\NN$, there exists a compact set $E_{k}\subset E$, such that $|E\setminus E_{k}|_{0}<1/k$, and for any $a>0$, there is a constant $c_{a,k}>0$, so that $|u(x,y)|\leq c_{a,k}$, $(x,y)\in\cup_{x^{0}\in E_{k}}\Gamma^{2}_{a}(x^{0})$. If we put $E_0=\cup_{k=1}^{\infty}E_k$, then $|E\setminus E_{0}|_{0}=0$. Since $E$ is $G$-invariant, we may choose each $E_k$ preserving this property; and otherwise, $\cap_{\sigma\in G}(\sigma E_k)$ could be used instead of $E_k$ and $\left|E\setminus\left[\cap_{\sigma\in G}(\sigma E_k)\right]\right|_0<|G|/k$. Thus applying Lemma \ref{fatou-c} to $u/c_{a,k}$, it follows that, for almost every $x\in E_k$, $u(t,y)/c_{a,k}$ has a finite limit as $(t,y)\in\Gamma_a(x)$ approaching to $(x,0)$, and hence, $u$ has a finite non-tangential limit at $(x,0)$ for almost every $x\in E$. The proof is completed. $\square$

\section{The proof of Theorem \ref{thm-2}}

In order to prove Theorem \ref{thm-2} (and also Theorems \ref{thm-3} and \ref{thm-4} in the next section), we introduce a variant of the area integral $S_{a,h}u$, that is,
\begin{align}\label{area-1}
(S^{\psi}_{a,h}u)(x)=\left(\iint_{\RR^d\times(0,h]}\left[\tau_{x}(\Delta_{\kappa}u^{2})\right](t,y)\,
\psi\left(\frac{t}{ay}\right)y^{1-d-2|\kappa|}\,d\omega_{\kappa}(t)dy\right)^{1/2},
\end{align}
where $\psi\in C^{\infty}(\RR^d)$ is radial and satisfies
$0\le\psi(x)\le1$,
\begin{align*}
\psi(x)=1\quad \hbox{for}\,\,|x|\le1/2,\quad\hbox{and}\quad \psi(x)=0\quad \hbox{for}\,\,|x|\ge1.
\end{align*}
The use of $\psi$ is to ensure that various operations on the generalized translation $\tau_x$ are more legitimate. Otherwise, we will often face the operation of this operator on some discontinuous functions such as the characteristic functions of balls, which will bring confusion or complexity in some cases.


We shall prove the following theorem, which implies Theorem \ref{thm-2} by Proposition \ref{area-positivity-a} below.

\begin{theorem}\label{thm-4-a}
Suppose that $E$ is a $G$-invariant measurable subset of $\partial\RR^{d+1}_+=\RR^d$ and $u$ is $\kappa$-harmonic in $\RR^{d+1}_+$. If $u$ is non-tangentially bounded at $(x,0)$ for every $x\in E$, then for $a,h>0$, the area integral $(S^{\psi}_{a,h}u)(x)$ is finite for almost every $x\in E$
\end{theorem}

Since the claim $(S_{a,h}u)(x)<\infty$ in Theorem \ref{thm-2} allows the truncated cones to be different at different points, the above theorem is a stronger version then Theorem \ref{thm-2}.

\begin{lemma}\label{Laplace-4-a}
If $u$ is $\kappa$-harmonic in $\RR^{d+1}_+$, then
\begin{eqnarray}\label{Laplace-4-1}
\left(\Delta_\kappa u^2\right)(x,y)=2\left|\nabla u(x,y)\right|^2+2\sum_{\alpha\in R_+}\kappa(\alpha)\left(\frac{u(x,y)-u(\sigma_{\alpha}(x),y)}{\langle\alpha,x\rangle}\right)^2,
\end{eqnarray}
where $\nabla=\nabla^{(x,y)}$ is the $(d+1)$-dimensional gradient.
\end{lemma}

The lemma is known in some literature. Indeed, since $u$ is $\kappa$-harmonic, from (\ref{Laplace-2-1}) and (\ref{Laplace-2-2}) direct calculations show that
\begin{align*}
\Delta u^2=2\left|\nabla^{(x,y)} u\right|^2-4u\sum_{\alpha\in R_+}\kappa(\alpha)\delta_\alpha u,
\end{align*}
and
\begin{align*}
2\sum_{\alpha\in R_+}\kappa(\alpha)\delta_{\alpha}u^2=2\sum_{\alpha\in R_+}\kappa(\alpha)\left(2u\frac{\langle\nabla^{(x)} u,\alpha\rangle}{\langle\alpha,x\rangle}-\frac{u^2-\sigma_{\alpha}u^2}
{\langle\alpha,x\rangle^2}\right).
\end{align*}
The equality (\ref{Laplace-4-1}) follows by summing the above two equalities.

\begin{proposition}\label{area-positivity-a}
If $u$ is $\kappa$-harmonic in $\RR^{d+1}_+$, then the area integrals $S_{a,h}u$ and $S^{\psi}_{a,h}u$ are well defined, and for $a,h>0$ and $x\in\RR^d$, we have
\begin{align}\label{area-integral-3-1}
\left(S^{\psi}_{a,h}u\right)(x)\le\left(S_{a,h}u\right)(x)\le\left(S^{\psi}_{2a,h}u\right)(x)
\end{align}
and
\begin{eqnarray}\label{area-4-2}
(S^{\psi}_{a,h}u)(x)=\left(\iint_{\RR^d\times(0,h]}\left(\Delta_{\kappa}u^{2}\right)(t,y)\,
\left(\tau_{-\frac{x}{ay}}\psi\right)\left(\frac{t}{ay}\right)
\frac{d\omega_{\kappa}(t)dy}{y^{2|\kappa|+d-1}}\right)^{1/2}.
\end{eqnarray}
\end{proposition}

\begin{proof}
Since $u$ is $\kappa$-harmonic in $\RR^{d+1}_+$, by Proposition \ref{harmonic-2-b} it is infinitely differentiable in $\RR^{d+1}_+$, so that for fixed $x\in\RR^d$, $\tau_{x}(\Delta_{\kappa}u^{2})\in C^{\infty}(\RR^{d+1}_+)$ by Proposition \ref{translation-2-b}(i). We write the integral in (\ref{area-1}) as
\begin{align}\label{area-integral-3-2}
\lim_{\delta\rightarrow0+}\iint_{\RR^d\times(\delta,h]}\left[\tau_{x}(\Delta_{\kappa}u^{2})\right](t,y)\,
\psi\left(\frac{t}{ay}\right)\frac{d\omega_{\kappa}(t)dy}{y^{2|\kappa|+d-1}},
\end{align}
and obviously, it can be further written as
\begin{align*}
\lim_{\delta\rightarrow0+}\int_{\delta}^h\int_0^{\infty}M_{(\Delta_{\kappa}u^2)(\cdot,y)}(x,r)
\psi_0\left(\frac{|t|}{ay}\right)\left(\frac{r}{y}\right)^{2|\kappa|+d-1}\,drdy,
\end{align*}
where $\psi(x)=\psi_0(|x|)$ for some $\psi_0\in C^{\infty}(\RR)$. Similarly the integral in (\ref{area-integral-1-1}) could be written into
\begin{align*}
\lim_{\delta\rightarrow0+}\int_{\delta}^h\int_0^{ay}M_{(\Delta_{\kappa}u^2)(\cdot,y)}(x,r)
\left(\frac{r}{y}\right)^{2|\kappa|+d-1}\,drdy.
\end{align*}
The last two expressions imply that the integrals in (\ref{area-integral-1-1}) and (\ref{area-1}) both are nonnegative since the generalized spherical mean $M_f$ is positivity-preserving, and hence both $S_{a,h}u$ and $S^{\psi}_{a,h}u$ are well defined. The inequalities in (\ref{area-integral-3-1}) are immediate since $\psi_0(\cdot/ay)\le\chi_{(0,ay)}\le\psi_0(\cdot/2ay)$, and (\ref{area-4-2}) follows from (\ref{area-integral-3-2}) by Propositions \ref{translation-2-b-1}(ii) and \ref{translation-2-d}. $\square$
\end{proof}

\begin{lemma}\label{area-4-b}
Let $E$ be a $G$-invariant compact subset of $\partial\RR^{d+1}_+=\RR^d$ and $u$ a $\kappa$-harmonic function in $\RR^{d+1}_+$. Then for $a,h>0$,
\begin{eqnarray}\label{area-4-3}
\int_{E}(S^{\psi}_{a,h}u)^{2}(x)\,d\omega_{\kappa}(x)\leq c\iint_{\Omega}y(\Delta_{\kappa}u^{2})(t,y)\,d\omega_{\kappa}(t)dy
  \end{eqnarray}
whenever the right hand side above is finite, where $\Omega=\Omega^E(a,h)$ and $c=a^{2|\kappa|+d}\|\psi\|_{L_{\kappa}^1(\RR^d)}$.
\end{lemma}

\begin{proof}
If for $x\in E$ and $y\in(0,h]$, $(t,y)\notin\cup_{\sigma\in G}\Gamma_{a}^{h}(\sigma(x))$, then
$$
|t-\sigma(x)|\ge ay\qquad \hbox{for all}\,\,\, \sigma\in G,
$$
and so, for $|\xi|<1$,
$$
|\xi|<\min_{\sigma\in G}\left|\frac{t}{ay}-\sigma\left(\frac{x}{ay}\right)\right|.
$$
Thus by (\ref{translation-2-3}) and (\ref{translation-support-2-1}) it follows that
\begin{align}\label{psi-support-4-1}
\left(\tau_{-\frac{x}{ay}}\psi\right)\left(\frac{t}{ay}\right)=0\qquad \hbox{for}\,\,\, (t,y)\notin\cup_{\sigma\in G}\Gamma_{a}^{h}(\sigma(x)),
\end{align}
and then, from (\ref{area-4-2}) we have
\begin{eqnarray}\label{area-4-4}
\int_{E}(S^{\psi}_{a,h}u)^{2}(x)d\omega_{\kappa}(x)=
\iint_{\Omega}(\Delta_{\kappa}u^{2})(t,y)\,k_1(t,y)\,\frac{d\omega_{\kappa}(t)dy}{y^{2|\kappa|+d-1}},
\end{eqnarray}
where
\begin{align*}
k_1(t,y)=\int_E\left(\tau_{-\frac{x}{ay}}\psi\right)\left(\frac{t}{ay}\right)d\omega_{\kappa}(x).
\end{align*}
But by Propositions \ref{translation-2-b}(iv) and \ref{translation-2-e}(iii),
$$
k_1(t,y)\le\int_{\RR^d}\left(\tau_{\frac{t}{ay}}\psi\right)\left(-\frac{x}{ay}\right)d\omega_{\kappa}(x)
=(ay)^{2|\kappa|+d}\|\psi\|_{L_{\kappa}^1(\RR^d)}.
$$
Inserting this into (\ref{area-4-4}) yields (\ref{area-4-3}). $\square$
\end{proof}

\begin{lemma}\label{area-4-d}
Let $E$ be a $G$-invariant compact subset of $\partial\RR^{d+1}_+=\RR^d$, and for $a,h>0$, put $\Omega=\Omega^E(a,h)$. Then there exists a family of $G$-invariant regions $\{\Omega_{\epsilon}\}_{\epsilon\in(0,h/3)}$, with the following properties:

{\rm (i)} \,\, $\overline{\Omega_{\epsilon}}\subset\Omega$, and $\Omega_{\epsilon_{1}}\subset\Omega_{\epsilon_{2}}$ if $\epsilon_{2}<\epsilon_{1}$;

{\rm (ii)}\,\, $\Omega_{\epsilon} \rightarrow \Omega$ as $\epsilon\rightarrow 0+$ (i.e. $\cup \Omega_{\epsilon}=\Omega)$;

{\rm (iii)} the boundary $\partial \Omega_{\epsilon}$ is the union of two parts, $\partial \Omega_{\epsilon}={\mathcal C}_\epsilon^1\cup{\mathcal C}_\epsilon^2$, so that ${\mathcal C}_\epsilon^2$ is a portion of the hyperplane $y=h-\epsilon$; and

{\rm (iv)} \  ${\mathcal C}_\epsilon^1$ is a portion of the hypersurface $y=a^{-1}\delta_{\epsilon}(x)$ where $\delta_\epsilon\in
C^\infty(\RR^d)$, and $|\partial_j\delta_\epsilon(x)|\le 1$, $j=1,\dots,d$.
\end{lemma}

The lemma is an analog of \cite[p. 206, Lemma 2.2.1]{St2}, and we only need to check the $G$-invariance of $\Omega_{\epsilon}$. As in \cite[pp. 206-7]{St2}, set $\delta(x)={\mathrm dist}(x,E)$. Then $\delta$ is a Lipschitz function, and since $E$ is $G$-invariant, so is $\delta$. We choose a nonnegative radial $\varphi\in C^{\infty}(\RR^d)$, satisfying ${\rm supp}\,\varphi\subset\overline{\BB(0,1)}$ and $\int_{\RR^d}\phi(x)dx=1$, and for $\epsilon\in(0,h/3)$, define $\delta_{\epsilon}(x)=\left(\delta\ast\varphi_{a\epsilon}\right)(x)+2a\epsilon$ and
$$
\Omega_{\epsilon}=\left\{(x,y):\,\delta_{\epsilon}(x)<ay,\,0<y<h-\epsilon\right\}.
$$
Obviously $\delta_{\epsilon}$ is $G$-invariant and so $\Omega_{\epsilon}$ is $G$-invariant. Since $\delta(x)<\delta_{\epsilon}(x)$, $\delta_{\epsilon_2}(x)<\delta_{\epsilon_1}(x)$ for $\epsilon_{2}<\epsilon_{1}$, and $\delta_{\epsilon}(x)$ tends to $\delta(x)$ as $\epsilon\rightarrow 0+$ uniformly, these $\Omega_{\epsilon}$'s satisfy the requirements (i)-(iv).

\begin{lemma}\label{area-4-e}  Let $b, \eta>0$ be given and $u$ a $\kappa$-harmonic function in $\cup_{\sigma\in G}\Gamma_{b}^{\eta}(\sigma(x^0))$. If $|u|\leq 1$ in $\cup_{\sigma\in G}\Gamma_{b}^{\eta}(\sigma(x^0))$, then for fixed $a\in(0,b)$ and $h\in(0,\eta)$, there exists a constant $c=c(a,b,h,\eta)>0$, so that $y|\nabla u|\le c$ in $\Gamma_a^h(x^0)$.
\end{lemma}

\begin{proof}
We apply Proposition \ref{harmonic-2-a} to the case when $\Omega=\cup_{\sigma\in G}\Gamma_{b}^{\eta}(\sigma(x^0))\subset\RR^{d+1}_+$, $\tilde{G}=G\otimes\ZZ_2$ with the multiplicity function $\tilde{\kappa}=(\kappa,0)$. In this case, for $(x,y)\in\Omega$, the associated generalized translation $\tilde{\tau}_{(x,y)}$ is given by $(\tilde{\tau}_{(x,y)}u)(t,s)=(\tau_xu)(t,y+s)$, where $\tau_x$ acts on the first argument as before. For $u$ is $\kappa$-harmonic in $\Omega$, one has $u\in C^{\infty}(\Omega)$ and
\begin{align}\label{spherical-mean-4-1}
u(x,y)=\tilde{M}_u((x,y),r)\quad \hbox{with}\,\,\,\tilde{M}_u((x,y),r)=\tilde{d}_{\kappa}\int_{\SB^{d}}(\tilde{\tau}_{(x,y)}u)(r(t',s'))W_{\kappa}(t')d(t',s')
\end{align}
provided $\overline{\BB((x,y),r)}\subset\Omega$ for $r>0$, where $d(t',s')$ denotes the area element on $\SB^{d}$, and $\tilde{d}_{\kappa}^{-1}=\int_{\SB^{d}}W_{\kappa}(t')d(t',s')$.

We take a nonnegative even $\phi_0\in C^{\infty}(\RR)$ with ${\rm supp}\,\phi_0\subseteq[-1,1]$ and decreasing on $[0,1]$, and define $\phi(x,y)=\phi_0(\sqrt{|x|^2+y^2})$ normalized so that $\int_{\RR^{d+1}}\phi(x,y)d\omega_{\kappa}(x)dy=1$, and $\phi_{\epsilon}(x,y)=\epsilon^{-2|\kappa|-d-1}\phi(\epsilon^{-1}x,\epsilon^{-1}y)$ for $\epsilon>0$. Thus for $(x,y)\in\Omega$, from (\ref{spherical-mean-4-1}) it follows that
\begin{align}\label{integral-mean-4-1}
u(x,y)=\int_{\RR^{d+1}}\left(\tilde{\tau}_{(x,y)}u\right)(t,s)\phi_{\epsilon}(t,s)\,d\omega_{\kappa}(t)ds
\end{align}
for $\epsilon>0$ so that $\overline{\BB((x,y),\epsilon)}\subset\Omega$, and moreover, by Proposition \ref{translation-2-d}(i),
\begin{align}\label{integral-mean-4-2}
u(x,y)=\int_{\RR^{d+1}}u(t,s)\left(\tilde{\tau}_{(-x,-y)}\phi_{\epsilon}\right)(t,s)\,d\omega_{\kappa}(t)ds.
\end{align}
Note that by (\ref{spherical-mean-2-2}), (\ref{spherical-mean-support-2-1}), and in view of (\ref{spherical-mean-4-1}),  the two integrals in (\ref{integral-mean-4-1}) and (\ref{integral-mean-4-2}) only use those values of $u$ in the set
\begin{align}\label{spherical-mean-support-4-1}
\bigcup_{\sigma\in G} \{(t,s)\in\RR^{d+1}:\,\,|(t,s)-(\sigma(x),y)|\le \epsilon\}.
\end{align}

Now for fixed $a\in(0,b)$ and $h\in(0,\eta)$ and for $(x,y)\in\Gamma_a^h(x^0)$, choose $\epsilon=c_1y$ for suitable $c_1>0$ depending only on $a,b,h$ and $\eta$, such that $\overline{\BB((x,y),\epsilon)}\subset\Gamma_b^{\eta}(x^0)$. Thus the set given in (\ref{spherical-mean-support-4-1}) is contained in $\Omega$, and for $j=1,2,\dots,d$, by the assumption we have
\begin{align}\label{integral-mean-4-3}
|(\partial_{x_j}u)(x,y)|\le\int_{\RR^{d+1}}
\left|\partial_{x_j}\left[\left(\tilde{\tau}_{(-x,-y)}\phi_{\epsilon}\right)(t,s)\right]\right|
d\omega_{\kappa}(t)ds.
\end{align}

From Proposition \ref{translation-2-b} and (\ref{translation-2-4}),
\begin{align}\label{translation-4-1}
\left(\tilde{\tau}_{(-x,-y)}\phi_{\epsilon}\right)(t,s)
=\int_{\RR^d}\phi_{0,\epsilon}\left(A(x,y,t,s,\xi)\right)\,d\mu^{\kappa}_{t}(\xi),
\end{align}
where $\phi_{0,\epsilon}(r)=\epsilon^{-2|\kappa|-d-1}\phi_0(\epsilon^{-1}r)$ and
$$
A(x,y,t,s,\xi)=\sqrt{|t|^2+|x|^2-2\langle x,\xi\rangle+(s-y)^2},
$$
and hence
\begin{align*}
\partial_{x_j}\left[\left(\tilde{\tau}_{(-x,-y)}\phi_{\epsilon}\right)(t,s)\right]
=\frac{1}{\epsilon^{2|\kappa|+d+2}}\int_{\RR^d}\frac{\phi'_0\left(\epsilon^{-1}A(x,y,t,s,\xi)\right)}
{A(x,y,t,s,\xi)}(x_j-\xi_j)\,d\mu^{\kappa}_{t}(\xi),
\end{align*}
But since $A(x,y,t,s,\xi)\ge\sqrt{|t|^2-|\xi|^2+|x-\xi|^2}\ge|x-\xi|$ by (\ref{intertwining-support-1}), one has \begin{align*}
\left|\partial_{x_j}\left[\left(\tilde{\tau}_{(-x,-y)}\phi_{\epsilon}\right)(t,s)\right]\right|
\le\epsilon^{-1}\left(\tilde{\tau}_{(-x,-y)}\tilde{\phi}_{\epsilon}\right)(t,s),
\end{align*}
where $\tilde{\phi}(x,y)=|\phi'_0(\sqrt{|x|^2+y^2})|$. Inserting this into (\ref{integral-mean-4-3}) and appealing to Proposition \ref{translation-2-e}(iii), we obtain
\begin{align*}
|(\partial_{x_j}u)(x,y)|\le\epsilon^{-1}\|\tilde{\phi}\|_{L^1(\RR^{d+1},d\omega_{\kappa}(t)ds)},
\end{align*}
and so, $y|(\partial_{x_j}u)(x,y)|\le c$ for $(x,y)\in\Gamma_a^h(x^0)$. Similarly, $y|(\partial_{y}u)(x,y)|\le c$ for $(x,y)\in\Gamma_a^h(x^0)$ and the lemma is proved. $\square$
\end{proof}

After preparing these lemmas, one may complete the proof of Theorem \ref{thm-4-a} along the way of \cite[pp. 208-9]{St2}.
\vskip .1in

{\it Proof of Theorem \ref{thm-4-a}}.

For given $a,h>0$, we fix $b>a$ and $\eta>h$. We may assume that $E$ is bounded, without loss of generality. By Lemma \ref{fatou-a}, for each $j\in\NN$, there exists a compact set $E_{j}\subset E$, such that $|E\setminus E_{j}|_{0}<1/j$, and there is a constant $c_{b,\eta,j}>0$, so that $|u(x,y)|\leq c_{b,\eta,j}$, $(x,y)\in\Omega^{E_j}(b,\eta)$. If we put $E_0=\bigcup_{j=1}^{\infty}E_j$, then $|E\setminus E_{0}|_{0}=0$. Since $E$ is $G$-invariant, we may choose each $E_j$ preserving this property as in the proof of Theorem \ref{bounded-convergence}.
Thus, the proof of the theorem would be completed once we prove that for a $G$-invariant compact set $E\subset\partial\RR_{+}^{d+1}=\RR^d$,
$$
\int_{E}(S^{\psi}_{a,h}u)^{2}(x)\,d\omega_{\kappa}(x)<\infty
$$
under the condition
\begin{eqnarray}\label{area-condition-1}
|u(x,y)|\leq 1\qquad \text{for}\,\,\,(x,y)\in\Omega^{E}(b,\eta);
\end{eqnarray}
and by Lemma \ref{area-4-b}, it suffices to show
\begin{eqnarray*}
\iint_{\Omega^E(a,h)}y(\Delta_{\kappa}u^{2})(t,y)\,d\omega_{\kappa}(t)dy<\infty.
\end{eqnarray*}
Moreover, for $\Omega=\Omega^E(a,h)$, there exists a family of $G$-invariant regions $\{\Omega_{\epsilon}\}_{\epsilon\in(0,h/3)}$ satisfying (i)-(iv) in Lemma \ref{area-4-d}, and hence, we only need to prove that, there exists some constant $c>0$ independent of $\epsilon\in(0,h/3)$, so that
\begin{eqnarray}\label{area-finite-2}
\iint_{\Omega_{\epsilon}}y(\Delta_{\kappa}u^{2})(t,y)\,d\omega_{\kappa}(t)dy\le c.
\end{eqnarray}
Taking integration over $\Omega_{\epsilon}$ instead of $\Omega$ is for legitimate use of Green's formula. In fact, for $U,V \in C^2(\overline{\Omega_{\epsilon}})$, the Green formula (\ref{Green-formula-2-1}) reads as
\begin{align*}
\iint_{\Omega_{\epsilon}}(V\Delta_{\kappa}U-U\Delta_{\kappa}V)\,d\omega_{\kappa}(x)dy
=\int_{\partial\Omega_{\epsilon}}\left(V\partial_{\mathbf{n}}U
-U\partial_{\mathbf{n}}V\right)W_{\kappa}(x)d\sigma(x,y).
\end{align*}
If $U=u^2$ and $V=y$, then
\begin{align}\label{Green-formula-2}
\iint_{\Omega_{\epsilon}}y(\Delta_{\kappa}u^2)(x,y)\,d\omega_{\kappa}(x)dy
=\int_{\partial\Omega_{\epsilon}}\left(y\frac{\partial u^2}{\partial \mathbf{n}}
-u^2\frac{\partial y}{\partial \mathbf{n}}\right)W_{\kappa}(x)d\sigma(x,y).
\end{align}

Since $\partial\Omega_{\epsilon}\subset\Omega^{E}(b,\eta)$, it follows from Lemma \ref{area-4-e}(i) and (\ref{area-condition-1}) that
\begin{eqnarray*}
\left|y\frac{\partial u^2}{\partial \mathbf{n}}\right|\leq 2y|u||\nabla u|\leq c\qquad \mathrm{on}\,\,\,\partial\Omega_{\epsilon};
\end{eqnarray*}
and since $\left|\partial y/\partial\mathbf{n}\right|\leq1$, we have $\left|u^2\partial y/\partial \mathbf{n}\right|\le1$ on $\partial\Omega_{\epsilon}$. Applying these estimates to the right hand side of
(\ref{Green-formula-2}) and on account of boundedness of $\Omega=\Omega^E(a,h)$, we get
\begin{align}\label{area-finite-3}
\iint_{\Omega_{\epsilon}}y(\Delta_{\kappa}u^2)(x,y)\,d\omega_{\kappa}(x)dy
\le c\int_{\partial\Omega_{\epsilon}}W_{\kappa}(x)d\sigma(x,y)
\le c'\int_{\partial\Omega_{\epsilon}}d\sigma(x,y).
\end{align}
Since $E$ is bounded, it follows from \cite[p. 209]{St2} that the area of $\partial\Omega_{\epsilon}$ is bounded by a constant independent of $\epsilon$. This proves (\ref{area-finite-2}) and end the proof of Theorem \ref{thm-4-a}. $\square$

\section{The proofs of Theorems \ref{thm-3} and \ref{thm-4}}

%
%

We shall prove the following analogs of Theorems \ref{thm-3} and \ref{thm-4} with $S^{\psi}_{a,h}u$ instead of $S_{a,h}u$. Theorems \ref{thm-3} and \ref{thm-4} follow immediately by Proposition \ref{area-positivity-a}.

\begin{theorem}\label{thm-6-3}
Suppose that $E$ is a $G$-invariant measurable subset of $\partial\RR^{d+1}_+=\RR^d$ and $u$ is $\kappa$-harmonic in $\RR^{d+1}_+$. If $u$ is $G$-invariant and for every $x\in E$,
its area integral $(S^{\psi}_{a,h}u)(x)$ is finite for some $a,h>0$, then $u$ is non-tangentially bounded at $(x,0)$ for almost every $x\in E$.
\end{theorem}

\begin{theorem}\label{thm-6-4}
Suppose $G=Z_{2}^{d}$ with a given multiplicity parameters $\lambda=(\lambda_1,\dots,\lambda_d)$, $E$ is a $G$-invariant measurable subset of $\RR^d$, and $u$ is $\lambda$-harmonic in $\RR^{d+1}_+$.
If  for every $x\in E$, the area integral $(S^{\psi}_{a,h}u)(x)$ is finite for some $a,h>0$, then $u$ is non-tangentially bounded at $(x,0)$ for almost every $x\in E$.
\end{theorem}

We need several lemmas.

\begin{lemma}\label{area-5-a}
Let $E$ be a $G$-invariant, bounded and measurable set of $\RR^{d}$. Then for any $\epsilon>0$, there exists a $G$-invariant compact set $E_{\epsilon}$ satisfying

{\rm(i)}\ \ $E_{\epsilon}\subset E$, $|E-E_{\epsilon}|_{\kappa}<\epsilon$;

{\rm(ii)} \ for $\theta\in(0,1)$, there exists some $\delta>0$, such that for $x\in E_{\epsilon}$ and $0<r<\delta$,
\begin{eqnarray*}
\left|\BB(x,r)\cap E\right|_{\kappa} >\theta|\BB(x,r)|_{\kappa}.
\end{eqnarray*}
\end{lemma}

\begin{proof}
By (\ref{ball-estimate-3}), $d\omega_{\kappa}$ is a doubling measure on $\RR^d$, and if $f$ is locally integrable with respect to $d\omega_{\kappa}$, then by \cite[p. 13, Corollary]{St3}, $\lim_{r\rightarrow0}|\BB(x,r)|_{\kappa}^{-1}\int_{\BB(x,r)}fd\omega_{\kappa}=f(x)$ for almost every $x\in\RR^d$. Taking $f=\chi_E$, it follows that
$$
\lim_{r\rightarrow0+}\frac{\left|\BB(x,r)\cap E\right|_{\kappa}}{|\BB(x,r)|_{\kappa}}=1\qquad \textrm{for a.e.}\,\, x\in E.
$$
By Egorov's theorem, for given $\epsilon>0$ there exists $E_{\epsilon}\subset E$, $|E-E_{\epsilon}|_{\kappa}<\epsilon$, so that $\left|\BB(x,r)\cap E\right|_{\kappa}/|\BB(x,r)|_{\kappa}$ tends to $1$ uniformly for $x\in E_{\epsilon}$ as $r\rightarrow0+$. Clearly one may take $E_{\epsilon}$ to be a closed subset of $E$. Since $E$ is $G$-invariant, the uniform convergence above is true also on $\sigma E_{\epsilon}$ for each $\sigma\in G$, and so is on
$\cup_{\sigma\in G}(\sigma E_{\epsilon})$. Thus we may further take $E_{\epsilon}$ to be $G$-invariant, as desired. $\square$
\end{proof}

The next lemma gives an equivalent estimate of $(\tau_{-x}\phi)(t)$ for radial $\phi\in C(\RR^d)$ with $\phi(0)>0$ and for $t$ close to $x$.

\begin{lemma}\label{phi-lower-bound-3-a}
Assume that $\phi\in C(\RR^d)$ is radial. If $\phi(0)>0$, then there exist some $c_1,c_2,\delta_0>0$, all independent of $x$, such that
\begin{align}\label{phi-lower-bound-3-1}
\frac{c_1}{|\BB(x,1)|_{\kappa}}\le(\tau_{-x}\phi)(t)
\le \frac{c_2}{|\BB(x,1)|_{\kappa}}\qquad \hbox{for}\,\,\,|t-x|<\delta_0.
\end{align}
\end{lemma}

We note that the first inequality in (\ref{phi-lower-bound-3-1}) is a sharper form of \cite[(4.15)]{ADH1}. Its proof depends on a distribution inequality for the representing measures of Dunkl's intertwining operator proved in \cite{JL1}.

\begin{lemma}\label{representing-measure-5-a} {\rm(\cite{JL1})}
There exists a constant $c>0$ independent of $\delta>0$ and $x\in\RR^d$ such that
\begin{align*}
\int_{\langle x,\xi\rangle>|x|^2-\delta^2}\,d\mu^{\kappa}_{x}(\xi)\ge
\frac{c\delta^{2|\kappa|+d}}{\left|\BB(x,\delta\right|_{\kappa}}.
\end{align*}
\end{lemma}

{\it Proof of Lemma \ref{phi-lower-bound-3-a}.}

By continuity there exists some $r_0>0$ so that $\phi(\xi)\ge\phi(0)/2$ whenever $|\xi|<r_0$. It follows from Proposition \ref{translation-2-b-2} and (\ref{translation-2-4}) that
\begin{align}\label{translation-5-0}
(\tau_{-x}\phi)(t)=(\tau_{x}\phi)(-t)=\int_{\RR^d}\phi_0(\sqrt{|x|^2+|t|^2-2\langle t,\xi\rangle})\,d\mu^{\kappa}_{x}(\xi),
\end{align}
where $\phi(x)=\phi_0(|x|)$ for some $\phi_0\in C(\RR)$.

If $\langle x,\xi\rangle>|x|^2-r_0^2/8$ for $\xi\in\hbox{supp}\,\mu^{\kappa}_{x}$, from (\ref{intertwining-support-1}) we have
\begin{align*}
|x-\xi|^2<|\xi|^2-|x|^2+r_0^2/4\le r_0^2/4,
\end{align*}
and so
\begin{align*}
|x|^2+|t|^2-2\langle t,\xi\rangle &=|x-t|^2-2\langle x-\xi,x-t\rangle+2|x|^2-2\langle x,\xi\rangle\\
&\le|x-t|^2+r_0|x-t|+r_0^2/4.
\end{align*}
Now we set $\delta_0=r_0/2$. It follows that $\sqrt{|x|^2+|t|^2-2\langle t,\xi\rangle}\le r_0$ whenever $|t-x|\le\delta_0$, and then, from (\ref{translation-5-0}),
\begin{align*}
(\tau_{-x}\phi)(t)
\ge\frac{\phi(0)}{2}\int_{\langle x,\xi\rangle>|x|^2-r_0^2/8}\,d\mu^{\kappa}_{x}(\xi).
\end{align*}
By Lemma \ref{representing-measure-5-a}, this proves the first inequality in (\ref{phi-lower-bound-3-1}) for $|t-x|\le\delta_0$.

The second inequality in (\ref{phi-lower-bound-3-1}) follows from \cite[Corollary 3.5]{DH1} since $|\BB(t,1)|_{\kappa}\asymp|\BB(x,1)|_{\kappa}$ when $|t-x|<\delta_0$. $\square$

\vskip .1in

\begin{lemma}\label{area-5-e}  Let $b, \eta>0$ be given and $u$ a $\kappa$-harmonic function on $\cup_{\sigma\in G}\Gamma_{b}^{\eta}(\sigma(x^0))$. If $(S^{\psi}_{b,\eta}u)(x^0)\le1$, then for fixed $a\in(0,b/2)$ and $h\in(0,\eta)$, there exists a constant $c=c(a,b,h,\eta)>0$, so that $y|\nabla_{\kappa} u|\le c$ in $\Gamma_a^h(x^0)$.
\end{lemma}

\begin{proof}
We use the notation in the proof of Lemma \ref{area-4-e}. Applying (\ref{integral-mean-4-2}) to $D_ju$ for $1\le j\le d$ instead of $u$ and noting that
$\int_{\RR^{d+1}}\left(\tilde{\tau}_{(-x,-y)}\phi_{\epsilon}\right)(t,s)d\omega_{\kappa}(t)ds=1$ by Proposition \ref{translation-2-e}(iii), we have
\begin{align}\label{integral-mean-4-4}
\left|(D_ju)(x,y)\right|^2\le\int_{\RR^{d+1}}\left|(D_ju)(t,s)\right|^2
\left(\tilde{\tau}_{(-x,-y)}\phi_{\epsilon}\right)(t,s)\,d\omega_{\kappa}(t)ds
\end{align}
for $\epsilon>0$ so that $\overline{\BB((x,y),\epsilon)}\subset\Omega:=\cup_{\sigma\in G}\Gamma_{b}^{\eta}(\sigma(x^0))$.

For fixed $a\in(0,b/2)$ and $h\in(0,\eta)$ and for $(x,y)\in\Gamma_a^h(x^0)$, take $\epsilon=c_2y$ for suitable $0<c_2<(b-2a)/(b+2)$, independent of $(x,y)$, such that $\overline{\BB((x,y),\epsilon)}\subset\Gamma_{b}^{\eta}(x^0)\subset\Omega$. We claim that, with the choice of $\epsilon$,
\begin{align}\label{translation-4-2}
\left(\tilde{\tau}_{(-x,-y)}\phi_{\epsilon}\right)(t,s)
\le \frac{cy^{-2}}{s^{2|\kappa|+d-1}}\left(\tau_{-\frac{x^0}{bs}}\psi\right)\left(\frac{t}{bs}\right).
\end{align}
Indeed, for $\xi\in\hbox{supp}\,\mu^{\kappa}_{t}$ satisfying $A(x,y,t,s,\xi)\le\epsilon$, one has $(1-c_2)y\le s\le(1+c_2)y\le\eta$, and
\begin{align*}
\sqrt{|t|^2+|x^0|^2-2\langle x^0,\xi\rangle}\le A(x,y,t,s,\xi)+|x-x^0|\le(c_2+a)y\le\frac{c_2+a}{1-c_2}s\le\frac{1}{2}bs.
\end{align*}
Thus for $\xi\in\hbox{supp}\,\mu^{\kappa}_{t}$,
\begin{align*}
\phi_{0}\left(\frac{1}{\epsilon}A(x,y,t,s,\xi)\right)\le c'\psi_0\left(\frac{1}{bs}\sqrt{|t|^2+|x^0|^2-2\langle x^0,\xi\rangle}\right)
\end{align*}
with $\psi_0$ such that $\psi(x)=\psi_0(|x|)$,
and so, from (\ref{translation-2-4}), (\ref{translation-4-1}) and Proposition \ref{translation-2-b-1}(ii),
\begin{align*}
\left(\tilde{\tau}_{(-x,-y)}\phi_{\epsilon}\right)(t,s)
\le \frac{c'}{\epsilon^{2|\kappa|+d+1}}\left(\tau_{-x^0}\psi\left(\frac{\cdot}{bs}\right)\right)(t)\le \frac{c''y^{-2}}{s^{2|\kappa|+d-1}}\left(\tau_{-\frac{x^0}{bs}}\psi\right)\left(\frac{t}{bs}\right),
\end{align*}
which verifies the claim (\ref{translation-4-2}).

Applying (\ref{translation-4-2}) to (\ref{integral-mean-4-4}) we obtain
\begin{align}\label{integral-mean-4-5}
y^2\left|(D_ju)(x,y)\right|^2\le c\int_{\RR^{d}\times(0,\eta]}\left|(D_ju)(t,s)\right|^2
\left(\tau_{-\frac{x^0}{bs}}\psi\right)\left(\frac{t}{bs}\right)
\frac{d\omega_{\kappa}(t)ds}{s^{2|\kappa|+d-1}}.
\end{align}
From (\ref{psi-support-4-1}), $\left(\tau_{-\frac{x^0}{bs}}\psi\right)\left(\frac{t}{bs}\right)=0$ for $(t,s)\notin\Omega$, and since $u$ is $\kappa$-harmonic in $\Omega$, (\ref{Dunkl-operator-1}) and (\ref{Laplace-4-1}) imply
$\left|(D_ju)(t,s)\right|^2\le c\left(\Delta_\kappa u^2\right)(t,s)$ for $(t,s)\in\Omega$. Thus from (\ref{integral-mean-4-5}), $y\left|(D_ju)(x,y)\right|\le c(S^{\psi}_{b,\eta}u)(x^0)\le c$ for $(x,y)\in\Gamma_a^h(x^0)$, and so does $y\left|(\partial_y u)(x,y)\right|$ similarly. The lemma is proved. $\square$
\end{proof}

\begin{lemma}\label{area-c}
Let $E$ be a $G$-invariant, bounded and measurable set of $\partial\RR^{d+1}_+=\RR^d$, $u$ a $\kappa$-harmonic function in $\RR^{d+1}_+$, and let $b, \eta>0$ be given. Then for any $\epsilon>0$, there exists a $G$-invariant compact set $E_{\epsilon}$ satisfying the following conditions:

{\rm(i)}\ \ $E_{\epsilon}\subset E$, $|E-E_{\epsilon}|_{\kappa}<\epsilon$;

{\rm(ii)} \ There are some $a\in(0,b)$, $h\in(0,\eta)$, and $c=c(\epsilon,a,b,h,\eta)>0$, such that
\begin{align}\label{key-inequality-5-1}
\iint_{\Omega^{E_{\epsilon}}(a,h)}y(\Delta_{\kappa}u^{2})(x,y)\,d\omega_{\kappa}(x)dy\le c\int_{E}(S^{\psi}_{b,\eta}u)^{2}(x)\,d\omega_{\kappa}(x),
\end{align}
whenever the right hand side above is finite.
\end{lemma}

\begin{proof}
First applying Lemma \ref{phi-lower-bound-3-a} to $\psi$, there exist some $c_0>0$, $\delta_0\in(0,1)$, both independent of $x$, such that
\begin{align}\label{psi-lower-bound-3-1}
(\tau_{-x}\psi)(t)
\ge \frac{c_0}{|\BB(x,1)|_{\kappa}}\qquad \hbox{for}\,\,|t-x|<\delta_0.
\end{align}
For given $\epsilon>0$, by Lemma \ref{area-5-a} there there exist a $G$-invariant compact subset $E_{\epsilon}$ of $E$ and a number $\delta_1>0$, satisfying $|E-E_{\epsilon}|_{\kappa}<\epsilon$ and
\begin{eqnarray}\label{fatou-a-2}
\left|\BB(x,r)\cap E\right|_{\kappa} >\frac{1}{2}|\BB(x,r)|_{\kappa}\qquad \hbox{for}\,\,x\in E_{\epsilon}\,\,\hbox{and}\,\,r\in(0,\delta_1).
\end{eqnarray}

Now we fix $a<\delta_0 b$ and $h<\eta$. Since $(t,y)\in\cup_{\sigma\in G}\Gamma_{b}^{\eta}(\sigma(x))$ is equivalent to $(x,y)\in\cup_{\sigma\in G}\Gamma_{b}^{\eta}(\sigma(t))$, from (\ref{area-4-2}) and (\ref{psi-support-4-1}) and by Propositions \ref{translation-2-b-1}(ii), we have
\begin{align}\label{area-inequality-1}
\int_{E}(S^{\psi}_{b,\eta}u)^{2}(x)\,d\omega_{\kappa}(x)
&=\int_{E}\iint_{\RR^d\times(0,\eta]}\left(\Delta_{\kappa}u^{2}\right)(t,y)\,
\left(\tau_{-\frac{x}{by}}\psi\right)\left(\frac{t}{by}\right)
\frac{d\omega_{\kappa}(t)dy}{y^{2|\kappa|+d-1}}\,d\omega_{\kappa}(x)\nonumber\\
 &\ge\iint_{\Omega^{E_{\epsilon}}(a,h)}(\Delta_{\kappa}u^{2})(t,y)\,
k_2(t,y)\,\frac{d\omega_{\kappa}(t)dy}{y^{2|\kappa|+d-1}},
\end{align}
where
\begin{eqnarray}\label{k-2-1}
k_2(t,y)=\int_{E}\left(\tau_{-\frac{x}{by}}\psi\right)\left(\frac{t}{by}\right)
\chi_{\cup_{\sigma\in G}\Gamma_{b}^{\eta}(\sigma(t))}(x,y)d\omega_{\kappa}(x).
\end{eqnarray}

For given $(t,y)\in\Omega^{E_{\epsilon}}(a,h)$, there exists some ${\bar x}\in E_{\epsilon}$ such that $(t,y)\in\Gamma_{a}^{h}({\bar x})$. Thus, when
\begin{eqnarray*}
|x-\bar{x}|<\gamma y,\qquad \gamma=\min\{(\delta_0 b-a)/2,\delta_1/h\},
\end{eqnarray*}
we have $|x-t|<a'y$ with $a'=(a+\delta_0 b)/2$, which certainly implies that $(x,y)\in\Gamma_{b}^{\eta}(t)$.

Since for $(x,y)\in\Gamma_{a'}^{h}(t)$, $|\frac{t}{by}-\frac{x}{by}|<a'y/by<\delta_0$, by (\ref{ball-estimate-3}) and (\ref{psi-lower-bound-3-1}) we have
\begin{eqnarray}\label{chi-1}
\left(\tau_{-\frac{x}{by}}\psi\right)\left(\frac{t}{by}\right)
\ge \frac{c\,y^{2|\kappa|+d}}{\left|\BB(x,by)\right|_{\kappa}}
\ge \frac{c\,y^{2|\kappa|+d}}{\left|\BB(\bar{x},\gamma y)\right|_{\kappa}}
\end{eqnarray}
whenever $|x-\bar{x}|<\gamma y$. Thus from (\ref{k-2-1}) and (\ref{chi-1}) it follows that
\begin{align*}
k_2(t,y)&\ge\int_{\BB(\bar{x},\gamma y)\cap E}\left(\tau_{-\frac{x}{by}}\psi\right)\left(\frac{t}{by}\right)
d\omega_{\kappa}(x)\\
&\ge\frac{c\,y^{2|\kappa|+d}}{\left|\BB(\bar{x},\gamma y)\right|_{\kappa}}\left|\BB(\bar{x},\gamma y)\cap E\right|_{\kappa}.
\end{align*}
Since for $(t,y)\in\Omega^{E_{\epsilon}}(a,h)$, $\gamma y<\gamma h\le\delta_1$, we use (\ref{fatou-a-2}) to get $k_2(t,y)\ge(c/2)y^{2|\kappa|+d}$. Inserting this into (\ref{area-inequality-1}) proves (\ref{key-inequality-5-1}) immediately. The proof of the lemma is completed. $\square$
\end{proof}

\vskip .1in

{\it Proof of Theorem \ref{thm-6-3}}.

Again, we assume that $E$ is bounded without loss of generality. By the assumption, $E=\bigcup E_j$, where $E_j=\cap_{\sigma\in G}\left\{x\in E:\,(S^{\psi}_{j^{-1},j^{-1}}u)(\sigma(x))\le j\right\}$. Note that for each $j$, $E_j$ is $G$-invariant, and by Lemma \ref{area-c}, there exists a sequence of $G$-invariant compact subsets $E_{j,k}$ ($k=1,2,\cdots$) of $E_j$ such that $|E_j\setminus E_{j,k}|_{\lambda}<1/k$, and there are some $a\in(0,1/j)$, $h\in(0,1/j)$, and $c=c(a,h,j,k)>0$, such that
\begin{eqnarray*}
\iint_{\Omega^{E_{j,k}}(a,h)}y(\Delta_{\kappa}u^{2})(x,y)\,d\omega_{\kappa}(x)dy\le c\int_{E_j}(S^{\psi}_{j^{-1},j^{-1}}u)^{2}(x)\,d\omega_{\kappa}(x).
  \end{eqnarray*}
For $(S^{\psi}_{j^{-1},j^{-1}}u)(x)\le j$ ($x\in E_j$) and $E_j$ is a bounded set, the left hand side above is finite; and further, by Lemma \ref{area-5-e} (with $0<a<1/2j$) there exists some $c_1=c_1(a,h,j,k)>0$ such that $y|\nabla_{\kappa}u(x,y)|\le c_1$ for $(x,y)\in\Omega^{E_{j}}(a,h)$. Since $u$ is $G$-invariant, this is identical to $y|\nabla u(x,y)|\le c_1$ for $(x,y)\in\Omega^{E_{j}}(a,h)$.

Since $|E_j\setminus \bigcup_{k=1}^{\infty}E_{j,k}|_{\kappa}=0$, the proof of the theorem would be completed once we prove that, for a $G$-invariant compact set $E\subset\partial\RR^{d+1}_+=\RR^d$, $u$ is non-tangentially bounded at $(x,0)$ for almost every $x\in E$ under the conditions
\begin{eqnarray}\label{area-integral-c-6}
\iint_{\Omega^{E}(a,h)}y(\Delta_{\kappa}u^{2})(x,y)\,d\omega_{\kappa}(x)dy\le 1
  \end{eqnarray}
for some $a,h>0$, and
\begin{eqnarray}\label{gradient-1}
y|\nabla u(x,y)|\le 1\qquad {\mathrm for}\,\,\,(x,y)\in\Omega^{E}(a,h).
  \end{eqnarray}
Note that the conditions (\ref{area-integral-c-6}) and (\ref{gradient-1}) come from the normalization of the conclusions deduced in the previous paragraph.

We fix $a_1\in(0,a)$ and $h_1\in(0,h)$, and work with $\Omega_{\epsilon}$ ($\epsilon\in(0,h_1/3)$) given by Lemma \ref{area-4-d} associated with $\Omega=\Omega^{E}(a_1,h_1)$. By Green's formula (\ref{Green-formula-2}), (\ref{area-integral-c-6}) implies
\begin{eqnarray}\label{area-integral-c-7}
\int_{\partial\Omega_{\epsilon}}\left(y\frac{\partial u^2}{\partial \mathbf{n}}
-u^2\frac{\partial y}{\partial \mathbf{n}}\right)W_{\kappa}(x)d\sigma(x,y)\le 1.
\end{eqnarray}
From Lemma \ref{area-4-d}, the boundary $\partial\Omega_{\epsilon}$ of $\Omega_{\epsilon}$ consists of two parts, $\partial \Omega_{\epsilon}={\mathcal C}_\epsilon^1\cup{\mathcal C}_\epsilon^2$, where ${\mathcal C}_\epsilon^1$ is a portion of the smooth hypersurface $y=a_1^{-1}\delta_{\epsilon}(x)$ with $|\partial_j\delta_\epsilon(x)|\le 1$, $j=1,\dots,d$, and ${\mathcal C}_\epsilon^2$ is a portion of the hyperplane $y=h_1-\epsilon$.

We extract the term $\int_{{\mathcal C}_\epsilon^1}u^2\frac{\partial y}{\partial \mathbf{n}}W_{\kappa}(x)d\sigma(x,y)$ from (\ref{area-integral-c-7}), so that
\begin{eqnarray*}
-\int_{{\mathcal C}_\epsilon^1}u^2\frac{\partial y}{\partial \mathbf{n}}W_{\kappa}(x)d\sigma(x,y)\le \int_{{\mathcal C}_\epsilon^2}u^2\frac{\partial y}{\partial \mathbf{n}}W_{\kappa}(x)d\sigma(x,y)
-\int_{\partial\Omega_{\epsilon}}y\frac{\partial u^2}{\partial \mathbf{n}}W_{\kappa}(x)d\sigma(x,y)+1.
\end{eqnarray*}
Since $E$ is bounded and $2h_1/3\le y\le h_1$ for $(x,y)\in{\mathcal C}_\epsilon^2$, the first term on the right hand side above is bounded by a constant independent of $\epsilon$, and so is the contribution coming from $(x,y)\in{\mathcal C}_\epsilon^2$ in the second term. Thus we get
\begin{eqnarray}\label{area-integral-c-8}
-\int_{{\mathcal C}_\epsilon^1}u^2\frac{\partial y}{\partial \mathbf{n}}W_{\kappa}(x)d\sigma(x,y)\le
-\int_{{\mathcal C}_\epsilon^1}y\frac{\partial u^2}{\partial \mathbf{n}}W_{\kappa}(x)d\sigma(x,y)+c.
\end{eqnarray}
Since the hypersurface ${\mathcal C}_\epsilon^1$ is determined by the equation $\delta_{\epsilon}(x)-a_1 y=0$ and so the direction ${\mathbf n}$ is given by $(\nabla^{(x)}\delta_{\epsilon},-a_1)/\sqrt{|\nabla^{(x)}\delta_{\epsilon}|^2+a_1^2}$, we have $\partial y/\partial {\mathbf n}=-a_1/\sqrt{|\nabla^{(x)}\delta_{\epsilon}|^2+a_1^2}\le -a_1/\sqrt{d+a_1^2}$; and since $\Omega_{\epsilon}\subset\Omega^{E}(a_1,h_1)$ is at a positive distance from the $x$-space $\RR^d$, (\ref{gradient-1}) implies that $y\left|\partial u^2/\partial \mathbf{n}\right|=2y|u|\left|\partial u/\partial \mathbf{n}\right|\le2y|u||\nabla u|\le 2|u|$ on ${\mathcal C}_\epsilon^1$. Applying these estimates to (\ref{area-integral-c-8}) gives
\begin{align*}
\frac{a_1}{\sqrt{d+a_1^2}}\int_{{\mathcal C}_\epsilon^1}u^2\,W_{\kappa}(x)d\sigma(x,y)
&\le 2\int_{{\mathcal C}_\epsilon^1}|u|\,W_{\kappa}(x)d\sigma(x,y)+c\\
&\le c'\left(\int_{{\mathcal C}_\epsilon^1}u^2\,W_{\kappa}(x)d\sigma(x,y)\right)^{1/2}+c,
\end{align*}
where the last inequality is due to the fact that $\int_{{\mathcal C}_\epsilon^1}W_{\kappa}(x)d\sigma(x,y)$ is bounded by a constant independent of $\epsilon$ as in (\ref{area-finite-3}). This certainly implies that
\begin{align}\label{u-integrability-1}
\int_{{\mathcal C}_\epsilon^1}u^2(x,y)W_{\kappa}(x)d\sigma(x,y)\le c,
\end{align}
where $c>0$ is independent of $\epsilon$.

If we define $f_{\epsilon}(x)=|u(x,a_1^{-1}\delta_{\epsilon}(x))|$ for $x$ satisfying $(x,y)\in{\mathcal C}_\epsilon^1$ with some $y>0$, and $f_{\epsilon}(x)=0$ otherwise, then from (\ref{u-integrability-1}), one has, for $\epsilon\in(0,h_1/3)$,
\begin{align}\label{u-integrability-2}
\int_{\RR^d}|f_{\epsilon}(x)|^2\,d\omega_{\kappa}(x)\le
\int_{{\mathcal C}_\epsilon^1}u^2(x,y)W_{\kappa}(x)d\sigma(x,y)\le c.
\end{align}

Now for $(x,y)\in{\mathcal C}_\epsilon^1\subset\Omega^{E}(a_1,h_1)$ with $0<y<h_1/2$, one can choose a constant $c\in(0,1)$ independent of $(x,y)$, so that the ball $\BB((x,y),cy)\subset\Omega^{E}(a,h)$. It follows that, for $(t,s)\in \BB((x,y),cy)$,
$|u(x,y)-u(t,s)|\le cy\sup|\nabla u|$, where the supremum is taken over the line segment joining $(x,y)$ and $(t,s)$, and by (\ref{gradient-1}),
\begin{align}\label{key-inequality-5-2}
|u(x,y)|\le |u(t,s)|+c,\qquad (t,s)\in \BB((x,y),cy).
\end{align}

Since $|\partial_j\delta_\epsilon(x)|\le 1$ ($j=1,\dots,d$), for $(x,y),(t,s)\in{\mathcal C}_\epsilon^1$ one has
$$
|(t,s)-(x,y)|=\sqrt{|t-x|^2+a_1^{-2}|\delta_{\epsilon}(t)-\delta_{\epsilon}(x)|^2}\le \tilde{c}|t-x|
$$
with $\tilde{c}=\sqrt{1+a_1^{-2}d}$, so that
\begin{align}\label{surface-1}
\int_{{\mathcal C}_\epsilon^1\cap\BB((x,y),cy)}W_{\kappa}(t)d\sigma(t,s)
\ge\int_{\BB(x,c_1y)}W_{\kappa}(t)dt
=|\BB(x,c_1y)|_{\kappa},
\end{align}
where $c_1=c/\tilde{c}<c$. In addition, it follows from (\ref{Poisson-ker-estimate-2}) that
\begin{align}\label{Poisson-ker-estimate-3}
(\tau_x P_y)(-t)\ge c'/\left|\BB(x,(c+1)y)\right|_{\kappa},\qquad (t,s)\in \BB((x,y),cy).
\end{align}
We then take the surface integration to (\ref{key-inequality-5-2}) over ${\mathcal C}_\epsilon^1\cap\BB((x,y),cy)$ and apply (\ref{surface-1}) and (\ref{Poisson-ker-estimate-3}), to get
\begin{align}\label{harmonic-majorization-1}
|u(x,y)|&\le c_2\int_{{\mathcal C}_\epsilon^1\cap\BB((x,y),cy)}|u(t,s)|(\tau_xP_y)(-t)
W_{\kappa}(t)d\sigma(t,s)+c\nonumber\\
&\le c_3 v_{\epsilon}(x,y)+c
\end{align}
for $(x,y)\in{\mathcal C}_\epsilon^1$ with $0<y<h_1/2$,
where $v_{\epsilon}(x,y)$ is the $\kappa$-Poisson integral of  $f_{\epsilon}(x)$. Further, since $u$ has a bound on $\{(x,y)\in\partial \Omega_{\epsilon}:\, h_1/2\le y\le h_1-\epsilon\}$ independent of $\epsilon$, we could choose the constant $c$ suitably large, so that (\ref{harmonic-majorization-1}) is true for all $(x,y)\in\partial \Omega_{\epsilon}$.

Considering the function $U(x,y)=|u(x,y)|-c_3v_{\epsilon}(x,y)-c$, by Lemma \ref{subharmonicity-3-b} we have
$\Delta_{\kappa}U(x,y)=\Delta_{\kappa}|u(x,y)|\geq 0$
in the region where $U(x,y)>0$ which implies $|u(x,y)|>0$. Since $U|_{\partial \Omega_{\epsilon}} \leq 0$ from (\ref{harmonic-majorization-1}), by the maximum principle (Lemma \ref{max-principle}) we assert that (\ref{harmonic-majorization-1}) holds on the whole $\Omega_{\epsilon}$.

Finally, since $\{f_{\epsilon}:\, \epsilon\in(0,h_1/3)\}$ is a bounded set in $L_{\kappa}^{2}(\RR^d)$ by (\ref{u-integrability-2}), there exists a sequence $\{f_{\epsilon_{k}}\}_{k=1}^{\infty}$, so that $f_{\varepsilon_{k}}$ converges weakly to a
function $f\in L_{\kappa}^{2}(\RR^d)$ as $k\rightarrow\infty$; and in particular, if $v(x,y)$ denotes the $\kappa$-Poisson integral of $f$, then $v_{\epsilon_k}(x,y)$ converges pointwise to $v(x,y)$ for $(x,y)\in\RR^{d+1}_+$. Thus, since by Lemma \ref{area-4-d}, $\Omega_{\epsilon_k}$ approaches increasingly to $\Omega^{E}(a_1,h_1)$, we conclude from (\ref{harmonic-majorization-1}) that
\begin{align*}
|u(x,y)|\le c''v(x,y)+c, \qquad (x,y)\in \Omega^{E}(a_1,h_1).
\end{align*}
By Proposition \ref{Poisson-d}, $v(x,y)$ is non-tangentially bounded at $(x,0)$ for almost every $x\in\RR^d$, and hence, the same is true for $u(x,y)$ and for almost every $x\in E$. The proof of Theorem \ref{thm-6-3} is completed. $\square$
\vskip .1in

In order to prove Theorem \ref{thm-6-4}, we need the following lemma for $G=Z_2^d$, that is an analog of Lemma \ref{area-5-e}.

\begin{lemma}\label{area-5-f}  Suppose $G=Z_{2}^{d}$ with a given multiplicity parameters $\lambda=(\lambda_1,\dots,\lambda_d)$. Let $b, \eta>0$ be given and $u$ a $\lambda$-harmonic function on $\Omega=\cup_{\sigma\in G}\Gamma_{b}^{\eta}(\sigma(x^0))$. If $(S^{\psi}_{b,\eta}u)(\sigma(x^0))\le1$ for all $\sigma\in G$, then for fixed $a\in(0,b/2)$ and $h\in(0,\eta)$, there exists a constant $c=c(a,b,h,\eta)>0$, so that $y|\nabla u|\le c$ in $\Gamma_a^h(x^0)$.
\end{lemma}

\begin{proof}
By Lemma \ref{area-5-e}, it suffices to show that, for $j=1,2,\dots,d$,
\begin{align}\label{key-inequality-5-3}
y\left|\frac{u(x,y)-u(\sigma_j(x),y)}{x_j}\right|\le c,\qquad (x,y)\in\Gamma_a^h(x^0),
\end{align}
where $\sigma_j(x)=(x_1,\dots, -x_j,\dots,x_d)$.

We continue to use the notation in the proof of Lemma \ref{area-4-e}. From (\ref{integral-mean-4-2}) we have
\begin{align}\label{key-equality-5-1}
&u(x,y)-u(\sigma_j(x),y)\nonumber\\
=&\int_{\RR^{d+1}}u(t,s)\left[\left(\tilde{\tau}_{(-x,-y)}\phi_{\epsilon}\right)(t,s)
-\left(\tilde{\tau}_{(-x,-y)}\phi_{\epsilon}\right)(\sigma_j(t),s)\right]\,d\omega_{\lambda}(t)ds\nonumber\\
=&\frac{1}{2}\int_{\RR^{d+1}}(u(t,s)-u(\sigma_j(t),s)\left[\left(\tilde{\tau}_{(-x,-y)}\phi_{\epsilon}\right)(t,s)
-\left(\tilde{\tau}_{(-x,-y)}\phi_{\epsilon}\right)(\sigma_j(t),s)\right]\,d\omega_{\lambda}(t)ds.
\end{align}
Recall that the intertwining operator $V_{\lambda}$ for $G=Z_{2}^{d}$ has the form (cf. \cite{Du3,Xu1})
\begin{align}\label{intertwining-5-1}
(V_{\lambda}f)(x)=\int_{[-1,1]^d}f(x_1\theta_1,\dots,x_d\theta_d)\,dm_{\lambda}(\theta),
\end{align}
where $dm_{\lambda}(\theta)=dm_{\lambda_1}(\theta_1)\dots dm_{\lambda_d}(\theta_d)$, and for $1\le i\le d$, $dm_{\lambda_i}(\theta_i)=c_{\lambda_i}(1+\theta_i)(1-\theta_i^2)^{\lambda_i-1}d\theta_i$,  $c_{\lambda_i}=\Gamma(\lambda_i+1/2)/\Gamma(\lambda_i)\Gamma(1/2)$.
In view of (\ref{translation-2-4}) we have
\begin{align*}
&\left(\tilde{\tau}_{(-x,-y)}\phi_{\epsilon}\right)(t,s)
-\left(\tilde{\tau}_{(-x,-y)}\phi_{\epsilon}\right)(\sigma_j(t),s)\\
&=\frac{1}{\epsilon^{2|\lambda|+d+1}} \int_{[-1,1]^d}\left[
\phi_0\left(\frac{1}{\epsilon}B(x,y,t,s,\theta)\right) -\phi_0\left(\frac{1}{\epsilon}B(x,y,\sigma_j(t),s,\theta)\right)\right]
\,dm_{\lambda}(\theta),
\end{align*}
where
$$
B(x,y,t,s,\theta)=\sqrt{|x|^2+|t|^2-2(x_1t_1\theta_1+\dots+x_dt_d\theta_d)+(s-y)^2}.
$$
The change of variables $\theta_j\mapsto-\theta_j$ to the second integrand above yields
\begin{align*}
&\left(\tilde{\tau}_{(-x,-y)}\phi_{\epsilon}\right)(t,s)
-\left(\tilde{\tau}_{(-x,-y)}\phi_{\epsilon}\right)(\sigma_j(t),s)\\
&=\frac{2c_{\lambda_j}}{\epsilon^{2|\lambda|+d+1}} \int_{[-1,1]^{d-1}}
\int_{-1}^1\phi_0\left(\frac{1}{\epsilon}B(x,y,t,s,\theta)\right)\theta_j(1-\theta_j^2)^{\lambda_j-1}d\theta_j
\,d\widetilde{m_{\lambda}}(\theta),
\end{align*}
where $d\widetilde{m_{\lambda}}(\theta)=\prod_{k\neq j}dm_{\lambda_k}(\theta_k)$. Now we take integration by parts for $\theta_j$ to get
\begin{align}\label{key-inequality-5-4}
&\left|\left(\tilde{\tau}_{(-x,-y)}\phi_{\epsilon}\right)(t,s)
-\left(\tilde{\tau}_{(-x,-y)}\phi_{\epsilon}\right)(\sigma_j(t),s)\right|\nonumber\\
&\le\frac{c_{\lambda_j}\lambda_j^{-1}|x_jt_j|}{\epsilon^{2|\lambda|+d+2}} \int_{[-1,1]^{d-1}}
\int_{-1}^1\frac{\left|\phi'_0\left(\epsilon^{-1}B(x,y,t,s,\theta)\right)\right|}{B(x,y,t,s,\theta)}
(1-\theta_j^2)^{\lambda_j}d\theta_j
\,d\widetilde{m_{\lambda}}(\theta),
\end{align}

We consider the generalized translation $\tau^+_x$ in $\RR^d$ associated the parameters
$$
\lambda^+=(\lambda_1,\dots,\lambda_{j-1},\lambda_j+1,\lambda_{j+1},\dots,\lambda_d).
$$
As before, the associated translation $\tilde{\tau}^+_{(x,y)}$ for $(x,y)\in\RR^{d+1}$ is given by $(\tilde{\tau}^+_{(x,y)}u)(t,s)=(\tau^+_xu)(t,y+s)$, where $\tau^+_x$ acts on the first argument. Note that $\tilde{\phi}_0(r):=|\phi'_0(r)/r|$ is in $C^{\infty}(\RR)$ and even. If we set $\tilde{\phi}(x,y)=\tilde{\phi}_0(\sqrt{|x|^2+y^2})$ and $\tilde{\phi}_{\epsilon}(x,y)=\epsilon^{-2|\lambda^+|-d-1}\tilde{\phi}(\epsilon^{-1}x,\epsilon^{-1}y)$, then from (\ref{translation-2-4}), (\ref{intertwining-5-1}) and (\ref{key-inequality-5-4}),
\begin{align}\label{key-equality-5-0}
&\left|\left(\tilde{\tau}_{(-x,-y)}\phi_{\epsilon}\right)(t,s)
-\left(\tilde{\tau}_{(-x,-y)}\phi_{\epsilon}\right)(\sigma_j(t),s)\right|\nonumber\\
&\le\frac{|x_jt_j|}{2\lambda_j+1}\left[\left(\tilde{\tau}^+_{(-x,-y)}\tilde{\phi}_{\epsilon}\right)(t,s)
+\left(\tilde{\tau}^+_{(-\sigma_j(x),-y)}\tilde{\phi}_{\epsilon}\right)(t,s)\right].
\end{align}

In what follows we set
$$
\Psi_j(x,y)=\left|\frac{u(x,y)-u(\sigma_j(x),y)}{x_j}\right|.
$$
Inserting (\ref{key-equality-5-0}) into (\ref{key-equality-5-1}) gives
\begin{align*}
(4\lambda_j+2)\Psi_j(x,y)
\le \int_{\RR^{d+1}}\Psi_j(t,s)
\left[\left(\tilde{\tau}^+_{(-x,-y)}\tilde{\phi}_{\epsilon}\right)(t,s)
+\left(\tilde{\tau}^+_{(-\sigma_j(x),-y)}\tilde{\phi}_{\epsilon}\right)(t,s)\right]\,d\omega_{\lambda^+}(t)ds,
\end{align*}
and by H\"older's inequality and Proposition \ref{translation-2-e}(iii),
\begin{align*}
\Psi_j(x,y)^2\le c\int_{\RR^{d+1}}\Psi_j(t,s)^2
\left[\left(\tilde{\tau}^+_{(-x,-y)}\tilde{\phi}_{\epsilon}\right)(t,s)
+\left(\tilde{\tau}^+_{(-\sigma_j(x),-y)}\tilde{\phi}_{\epsilon}\right)(t,s)\right]\,d\omega_{\lambda^+}(t)ds,
\end{align*}
where $c=\|\tilde{\phi}\|_{L^1(\RR^{d+1},d\omega_{\lambda^+}(t)ds)}/2(2\lambda_j+1)^2$. Furthermore, by Proposition \ref{translation-2-d}(i),
\begin{align*}
\Psi_j(x,y)^2
\le c\int_{\RR^{d+1}}
\left[\left(\tilde{\tau}^+_{(x,y)}+\tilde{\tau}^+_{(\sigma_j(x),y)}\right)\Psi_j^2\right](t,s)
\tilde{\phi}_{\epsilon}(t,s)
\,d\omega_{\lambda^+}(t)ds,
\end{align*}
and since
$$
t_j^2\tilde{\phi}_{\epsilon}(t,s)
\le\epsilon^{-2|\lambda|-d-1}\Phi(\epsilon^{-1}t,\epsilon^{-1}s),
\qquad \Phi(t,s)=\sqrt{|t|^2+s^2}|\phi'_0(\sqrt{|t|^2+s^2})|,
$$
it follows that
\begin{align}\label{key-inequality-5-5}
\epsilon^{2|\lambda|+d+1}\Psi_j(x,y)^2
\le c\int_{\RR^{d+1}}
\left[\left(\tilde{\tau}^+_{(x,y)}+\tilde{\tau}^+_{(\sigma_j(x),y)}\right)\Psi_j^2\right](t,s)
\Phi\left(\frac{t}{\epsilon},\frac{s}{\epsilon}\right)
\,d\omega_{\lambda}(t)ds.
\end{align}

We note that the translation operation $\tilde{\tau}_{(x,y)}$ can be iteratively written into
$\tilde{\tau}_{(x,y)}=\tau_{x_1}\dots\tau_{x_d}\tau_y$,
where $\tau_y$ is the usual translation, and $\tau_{x_i}$ ($1\le i\le d$) denotes the one-dimensional generalized translation in the $x_i$ variable associated to the parameter $\lambda_i$, which is given by
(cf. \cite{LZ,Ro1})
\begin{eqnarray}\label{tau-2}
(\tau_{x_i}f)(t_i)=\int_{-1}^1\left(f_e(\tilde{B}_{x_i,t_i}(\theta_i))+ f_o(\tilde{B}_{x_i,t_i}(\theta_i))\frac{x_i+t_i}{\tilde{B}_{x_i,t_i}(\theta_i)}\right)dm_{\lambda_i}(\theta_i),
\end{eqnarray}
where $f$ is a function defined on $\RR^1$, $f_e$ and $f_o$ are its even part and odd part respectively, and $\tilde{B}_{x_i,t_i}(\theta_i)=\sqrt{x_i^2+t_i^2+2x_it_i\theta_i}$. If we use $\tau^+_{x_j}$ to denote the generalized translation in the $x_j$ variable associated to the parameter $\lambda_j+1$, then $\tilde{\tau}^+_{(x,y)}=\tau_{x_1}\dots\tau_{x_{j-1}}\tau^+_{x_j}\tau_{x_{j+1}}\dots\tau_{x_d}\tau_y$.
But from (\ref{tau-2}), one has
\begin{align*}
&\left[(\tau^+_{x_j}+\tau^+_{-x_j})f\right](t_j)
=c_{\lambda_j+1}\int_{-1}^1\left[f(\tilde{B}_{x_j,t_j}(\theta_j))
\left(1+\frac{t_j+x_j\theta_j}{\tilde{B}_{x_j,t_j}(\theta_j)}\right)\right.\\
&\qquad\qquad\qquad\qquad\qquad\qquad\qquad\left. +f(-\tilde{B}_{x_j,t_j}(\theta_j))\left(1-\frac{t_j+x_j\theta_j}{\tilde{B}_{x_j,t_j}(\theta_j)}\right)
\right] (1-\theta_j^2)^{\lambda_j}d\theta_j,
\end{align*}
which implies that $\tau^+_{x_j}+\tau^+_{-x_j}$ is a positive operator, and if $f$ is nonnegative, then
\begin{align}\label{key-inequality-5-6}
&\left[(\tau^+_{x_j}+\tau^+_{-x_j})f\right](t_j)\le\frac{2\lambda_j+1}{2\lambda_j}
\left[(\tau_{x_j}+\tau_{-x_j})f\right](t_j).
\end{align}

Now we rewrite (\ref{key-inequality-5-5}) as
\begin{align*}
\epsilon^{2|\lambda|+d+1}\Psi_j(x,y)^2
\le c\int_{\RR^{d+1}}
\left[\left(\tau_{x_j}^+ +\tau_{-x_j}^+\right)\Psi_j^2\right](t,s)
K_{\epsilon}(x,y,t,s)
\,d\omega_{\lambda}(t)ds,
\end{align*}
where $K_{\epsilon}(x,y,t,s)=\left[\tau_{-x_1}\dots\tau_{-x_{j-1}}\tau_{-x_{j+1}}\dots\tau_{-x_d}\tau_{-y}
\left(\Phi\left(\epsilon^{-1}\cdot\right)\right)\right](t,s)$, and applying (\ref{key-inequality-5-6}), we get
\begin{align}\label{key-inequality-5-7}
&\epsilon^{2|\lambda|+d+1}\Psi_j(x,y)^2
\le \frac{c(2\lambda_j+1)}{2\lambda_j}\int_{\RR^{d+1}}
\Psi_j(t,s)^2 \nonumber\\ &\qquad\qquad\times\left[\left(\tilde{\tau}_{(-x,-y)}\left(\Phi\left(\epsilon^{-1}\cdot\right)\right)\right)(t,s)
+\left(\tilde{\tau}_{(-\sigma_j(x),-y)}\left(\Phi\left(\epsilon^{-1}\cdot\right)\right)\right)(t,s)\right]
\,d\omega_{\lambda}(t)ds,
\end{align}

As in the proof of Lemma \ref{area-5-e}, for fixed $a\in(0,b/2)$ and $h\in(0,\eta)$ and for $(x,y)\in\Gamma_a^h(x^0)$, take $\epsilon=c_2y$ for suitable $0<c_2<(b-2a)/(b+2)$, independent of $(x,y)$, such that $\overline{\BB((x,y),\epsilon)}\subset\Gamma_{b}^{\eta}(x^0)\subset\Omega$. Similarly to (\ref{translation-4-2}) we have
\begin{align}
\left(\tilde{\tau}_{(-x,-y)}\left(\Phi\left(\epsilon^{-1}\cdot\right)\right)\right)(t,s)
&\le \frac{cy^{2|\lambda|+d-1}}{s^{2|\lambda|+d-1}}\left(\tau_{-\frac{x^0}{bs}}\psi\right)\left(\frac{t}{bs}\right),
\label{translation-5-1}\\
\left(\tilde{\tau}_{(-\sigma_j(x),-y)}\left(\Phi\left(\epsilon^{-1}\cdot\right)\right)\right)(t,s)
&\le \frac{cy^{2|\lambda|+d-1}}{s^{2|\lambda|+d-1}}\left(\tau_{-\frac{\sigma_j(x^0)}{bs}}\psi\right)\left(\frac{t}{bs}\right),
\label{translation-5-2}
\end{align}
and for $s>\eta$, both the quantities on the left hand side above vanish.

Applying (\ref{translation-5-1}) and (\ref{translation-5-2}) to (\ref{key-inequality-5-7}) we obtain
\begin{align}\label{key-inequality-5-8}
y^2\Psi_j(x,y)^2\le c\int_{\RR^{d}\times(0,\eta]}\Psi_j(t,s)^2
\left[\left(\tau_{-\frac{x^0}{bs}}\psi\right)\left(\frac{t}{bs}\right)
+\left(\tau_{-\frac{\sigma_j(x^0)}{bs}}\psi\right)\left(\frac{t}{bs}\right)\right]
\frac{d\omega_{\lambda}(t)ds}{s^{2|\lambda|+d-1}}.
\end{align}
From (\ref{psi-support-4-1}), $\left(\tau_{-\frac{x^0}{bs}}\psi\right)\left(\frac{t}{bs}\right)
=\left(\tau_{-\frac{\sigma_j(x^0)}{bs}}\psi\right)\left(\frac{t}{bs}\right)=0$ for $(t,s)\notin\Omega$, and since $u$ is $\lambda$-harmonic in $\Omega$, from (\ref{Laplace-4-1}) it follows that
$\Psi_j(t,s)^2\le (2\lambda_j)^{-1}\left(\Delta_\kappa u^2\right)(t,s)$ for $(t,s)\in\Omega$. Thus from (\ref{area-4-2}) and (\ref{key-inequality-5-8}),
$$
y^2\Psi_j(x,y)^2\le c\left[(S^{\psi}_{b,\eta}u)^2(x^0)+(S^{\psi}_{b,\eta}u)^2(\sigma_j(x^0))\right]\le 2c
$$
for $(x,y)\in\Gamma_a^h(x^0)$. Thus (\ref{key-inequality-5-3}) is proved and the proof of the lemma is finished. $\square$
\end{proof}

\vskip .1in

We remark that for general groups $G$, it is uncertain whether the analog to Lemma \ref{area-5-f} holds, because further properties of the intertwining operator are needed in its proof, which are not known yet.

Now one may complete the proof of Theorem \ref{thm-6-4} by following the line of that of Theorem \ref{thm-6-3}. It should be noted that, in the current situation ($G=Z_2^d$), by Lemma \ref{area-5-f} the condition (\ref{gradient-1}) is true
without $G$-invariance of $u$ under the assumptions of Theorem \ref{thm-6-4}.

{\small

\section*{Declarations}

\subsection*{Funding}

The research is supported by the National Natural Science Foundation of China (No. 12071295).

\subsection*{Conflicts of interest/Competing interests}

  None of the authors have any conflicts of interest/competing interests in the manuscript.

\subsection*{Availability of data and materials}

Not applicable.

\subsection*{Code availability}

Not applicable.

\subsection*{Author's contributions}

    All authors contributed equally and significantly in conducting this research work and writing this paper. All authors read and approved the final manuscript.

}

\end{document}